\documentclass[reqno,english]{amsart}
\usepackage{amsfonts,amsmath,latexsym,verbatim,amscd,mathrsfs,color,array}
\usepackage[colorlinks=true]{hyperref}
\usepackage{amsmath,amssymb,amsthm,amsfonts,graphicx,color}
\usepackage{amssymb}
\usepackage{pdfsync}
\usepackage{epstopdf}
\usepackage{cite}

\usepackage{graphicx}

\newcommand{\cuad}{{\sqcap\kern-.68em\sqcup}}

\newcommand{\be}{\begin{equation}}
\newcommand{\ee}{\end{equation}}

\newtheorem{lemma}{Lemma}[section]
\newtheorem{prop}{Proposition}[section]
\newtheorem{theorem}{Theorem}

\newtheorem{remark}{Remark}[section]
\newcommand{\bremark}{\begin{remark} \em}
\newcommand{\eremark}{\end{remark} }

\long\def\comment#1{\marginpar{\vtop{\raggedright\small$\bullet$\ #1}}}
\long\def\hide#1{}

\long\def\anot#1{\ \\{\bf \crr ANNOTATION.} {#1}}
\long\def\noanot#1{}
\definecolor{redd}{RGB}{200,0,0}

\def\crr{\color{red}}
\long\def\elim#1{{\color{red} ELIMINAR\\ #1}}

\def\crr{}
\long\def\comment#1{}
\long\def\anot#1{}
\long\def\elim#1{}

\long\def\comment#1{\marginpar{\raggedright\small$\bullet$\ #1}}
\numberwithin{equation}{section}
\title[singularities in the harmonic map flow with free boundary]{Singularity formation in the harmonic map flow with free boundary}
\author[Y. Sire]{Yannick Sire}
\address{\noindent Department of Mathematics, Johns Hopkins University, 404 Krieger Hall, 3400 N. Charles Street, Baltimore, MD 21218, USA}
\email{sire@math.jhu.edu}
\author[J. Wei]{Juncheng Wei}
\address{\noindent Department of Mathematics, University of British Columbia, Vancouver, B.C., Canada, V6T 1Z2}
\email{jcwei@math.ubc.ca}
\author[Y. Zheng]{Youquan Zheng}
\address{\noindent School of Mathematics, Tianjin University, Tianjin 300072, P. R. China}
\email{zhengyq@tju.edu.cn}
\begin{document}
\begin{abstract}
In the past years, there has been a new light shed on the harmonic map problem with free boundary in view of its connection with nonlocal equations. Here we fully exploit this link, considering the harmonic map flow with free boundary
\begin{equation}\label{e:main0}
\begin{cases}
u_t = \Delta u\text{ in }\mathbb{R}^2_+\times (0, T),\\
u(x,0,t) \in \mathbb{S}^1\text{ for all }(x,0,t)\in \partial\mathbb{R}^2_+\times (0, T),\\
\frac{du}{dy}(x,0,t)\perp T_{u(x,0,t)}\mathbb{S}^1\text{ for all }(x,0,t)\in \partial\mathbb{R}^2_+\times (0, T),\\
u(\cdot, 0) = u_0\text{ in }\mathbb{R}^2_+
  \end{cases}
\end{equation}
for a function $u:\mathbb{R}^2_+\times [0, T)\to \mathbb{R}^2$. Here $u_0 :\mathbb{R}^2_+\to \mathbb{R}^2$ is a given smooth map and $\perp$ stands for orthogonality. We prove the existence of initial data $u_0$ such that (\ref{e:main0}) blows up at finite time with a profile being the half-harmonic map. This answers a question raised by Yunmei Chen and Fanghua Lin in Remark 4.9 of \cite{ChenLinJGA1998}.
\end{abstract}
\maketitle
\section{Introduction}
Let $(M, g)$ be an $m$-dimensional Riemannian manifold with boundary $\partial M$ and $N$ be an $l$-dimensional manifold without boundary. Suppsoe $\Sigma$ is a $k$-dimensional submanifold in $N$ without boundary. Any continuous map $u_0: M\to N$ satisfying $u_0(\partial M)\subset \Sigma$ defines a relative homotopy class in maps from $(M, \partial M)$ to $(N, \Sigma)$. A map $u: M\to N$ with $u(\partial M)\subset \Sigma$ is called homotopic to $u_0$ if there exist a continuous homotopy $h:[0, 1]\times M \to N$ satisfying $h([0,1]\times \partial M)\subset \Sigma$, $h(0) = u_0$ and $h(1) = u$. An interesting problem is that whether or not each relative homotopy class of maps has a representation by harmonic maps, which is equivalent to the following problem,
\begin{equation}\label{e:harmonicmapwithfreeboundary}
\begin{cases}
-\Delta u = \Gamma(u)(\nabla u, \nabla u),\\
u(\partial M)\subset \Sigma,\\
\frac{\partial}{\partial \nu}u\perp T_u\Sigma.
  \end{cases}
\end{equation}
Here $\nu$ is the unit normal vector of $M$ along the boundary $\partial M$, $\Delta \equiv \Delta_M$ is Laplace-Beltrami operator of $M$, $\Gamma$ is the second fundamental form of $N$ (viewed as a submanifold in $\mathbb{R}^n$), $T_pN$ is the tangent space in $\mathbb{R}^n$ of $N$ at $p$ and $\perp$ means orthogonal in $\mathbb{R}^n$. (\ref{e:harmonicmapwithfreeboundary}) is the Euler-Lagrangian equation for critical points of the following energy functional
\begin{equation*}
E(u) = \int_{M}|\nabla u|^2dM
\end{equation*}
defined on the space of maps
\begin{equation*}
H^1_\Sigma(M, N) = \{u\in H^1(M, N): u(\partial M)\subset \Sigma\}.
\end{equation*}
Here $H^1(M, N)$ is the usual Sobolev space of maps $u: M\to N$ satisfying $\nabla u\in L^2$. Existence results and partial regularity of energy minimizing maps on $H^1_\Sigma(M, N)$ were established (for example) in \cite{BaldesMM1982}, \cite{DuzaarSteffen1989JRAM}, \cite{DuzaarSteffenAA1989}, \cite{GulliverJostJRAM1987}, \cite{HardtLinCPAM1989}. A classical method for (\ref{e:harmonicmapwithfreeboundary}) is to study the following parabolic problem
\begin{equation}\label{e:harmonicmapflowwithfreeboundary}
\begin{cases}
\partial_t u -\Delta u = \Gamma(u)(\nabla u, \nabla u)\text{ on }M\times [0, \infty),\\
u(x, t)\in \Sigma\text{ on }x\in \partial M, t\geq 0,\\
\frac{\partial}{\partial \nu}u(x, t)\perp T_{u(x, t)}\Sigma \text{ for } x\in \partial M, t\geq 0,\\
u(\cdot, 0) = u_0\text{ on }M.
  \end{cases}
\end{equation}
This is the so-called harmonic map flow with free boundary. (\ref{e:harmonicmapflowwithfreeboundary}) was first studied by Ma \cite{MaLiCMH1991} in the case $m = dim M = 2$, where a global existence and uniqueness result for finite energy weak solutions were obtained under geometrical hypotheses on $N$ and $\Sigma$. Global existence theorem for weak solutions of (\ref{e:harmonicmapflowwithfreeboundary}) were also established by Struwe in \cite{StruweManMath1991}. In \cite{Hamilton1975LNM}, Hamilton considered the case when $\partial N = \Sigma$ is totally geodesic and $k_N\leq 0$.
He proved the global existence of a classical solution for (\ref{e:harmonicmapflowwithfreeboundary}). When $N = \mathbb{R}^n$, (\ref{e:harmonicmapflowwithfreeboundary}) is the standard heat equation
\begin{equation*}
u_t - \Delta u = 0 \text{ on }M \times [0, \infty).
\end{equation*}
As pointed out in \cite{ChenLinJGA1998} and \cite{StruweManMath1991}, estimates near the boundary for (\ref{e:harmonicmapflowwithfreeboundary}) are quite difficult due to the high nonlinearity of the boundary conditions. In \cite{jost2016qualitative}, Jost, Liu and Zhu showed that the energy identity at finite singular time as well as at infinity time and the no-neck property holds at infinite time in a blowing-up process for (\ref{e:harmonicmapflowwithfreeboundary}).

In the seminal paper \cite{ChenLinJGA1998} by Chen and Lin, the blow-up phenomenon for harmonic map flow with free boundary problem was studied,
where the authors gave many blow-up examples in higher dimensions and also a blowing-up theorem. In low dimensions, they asked the following question: {\em ``When $M$ is a smooth domain in $\mathbb{R}^2$, $N = \mathbb{R}^n$ and $\Sigma$ a smooth compact submanifold of $\mathbb{R}^n$, is there is a smooth initial datum $u_0$ such that (\ref{e:harmonicmapflowwithfreeboundary}) has no global smooth solutions ?" }.  In this paper we answer this question affirmatively.
More precisely we consider the problem (\ref{e:harmonicmapflowwithfreeboundary}) when $M = \mathbb{R}^2_+$ and $\Sigma = \mathbb{S}^1\subset\mathbb{R}^2$, i.e.  the following parabolic equation
\begin{equation}\label{e:main}
\begin{cases}
u_t = \Delta u\text{ in }\mathbb{R}^2_+\times (0, T),\\
u(x,0,t) \in \mathbb{S}^1\text{ for all }(x,0,t)\in \partial\mathbb{R}^2_+\times (0, T),\\
-\frac{du}{dy}(x,0,t)\perp T_{u(x,0,t)}\mathbb{S}^1\text{ for all }(x,0,t)\in \partial\mathbb{R}^2_+\times (0, T),\\
u(\cdot, 0) = u_0\text{ in }\mathbb{R}^2_+
  \end{cases}
\end{equation}
for a function $u:\mathbb{R}^2_+\times [0, T)\to \mathbb{R}^2$. Here $u_0 :\mathbb{R}^2_+\to \mathbb{R}^2$ is a given smooth map and $\perp$ stands for orthogonality.

The stationary solution of (\ref{e:main}) $u:\mathbb{R}^2_+\to \mathbb{R}^2$ satisfies
\begin{equation}\label{e:halfharmonicmap}
\begin{cases}
\Delta u = 0\text{ in }\mathbb{R}^2_+\times (0, T),\\
u(x,0) \in \mathbb{S}^1\text{ for all }(x,0)\in \partial\mathbb{R}^2_+:=\mathbb{R}\times \{0\},\\
-\frac{du}{dy}(x,0)\perp T_{u(x,0)}\mathbb{S}^1\text{ for all }(x,0)\in \partial\mathbb{R}^2_+.
  \end{cases}
\end{equation}
This is the harmonic extension form of the so-called half harmonic map from $\mathbb{R}$ into $\mathbb{S}^1$, which was systematically studied in
\cite{MillotSireARMA2015} and the nondegeneracy property was proved in \cite{sire2017nondegeneracy}. In particular, it was proved in \cite{MillotSireARMA2015} that
\begin{prop}
Let $u\in \dot{H}^{1/2}(\mathbb{R},\mathbb{S}^1)$ be a non-constant entire half-harmonic map and $u^e$ be its harmonic extension to $\mathbb{R}^2_+$ satisfying (\ref{e:halfharmonicmap}). There exist $d\in\mathbb{N}$, $\vartheta\in \mathbb{R}$, $\{\lambda_k\}_{k=1}^d\subset (0, \infty)$ and $\{a_k\}_{k=1}^d\subset \mathbb{R}$ such that $u^e(z)$ or its complex conjugate equals to
$$
e^{i\vartheta}\prod_{k=1}^d\frac{\lambda_k(z-a_k)-i}{\lambda_k(z-a_k)+i}.
$$
Furthermore, the energy can be expressed as
$$
\mathcal{E}(u,\mathbb{R}) = [u]^2_{H^{1/2}(\mathbb{R})}=\frac{1}{2}\int_{\mathbb{R}^2_+}|\nabla u^e|^2dz = \pi d.
$$
\end{prop}
This proposition indicates that the map $\omega:\mathbb{R}\to\mathbb{S}^1$
\begin{equation}\label{e:halfharmonicmapcanonical}
x\to \begin{pmatrix}
     \frac{2x}{x^2+1}\\
     \frac{x^2-1}{x^2+1}
\end{pmatrix}
\end{equation}
is a half-harmonic map which corresponds to the case $\vartheta = 0$, $d=1$, $\lambda_1=1$ and $a_1=0$. Notice that the previous equations involve nonlinear Neumann boundary conditions. This is a feature of nonlocal problems and as previously mentioned, we shall exploit this fact in a systematic way.
Our main result is
\begin{theorem}\label{t:main}
Given points $q = (q_1,\cdots, q_k)\in \left(\partial\mathbb{R}^2_+\right)^k:=(\mathbb{R}\times \{0\})^k$ and any sufficiently small $T > 0$, there exists $u_0$ such that the solution $u_q(x, t)$ of Problem (\ref{e:main}) blows-up at exactly those $k$ points as $t\nearrow T$. More precisely, there exist numbers $k^*_i > 0$ and a function $u_*\in H^1(\mathbb{R}^2_+)\cap C(\mathbb{R}^2_+)$ such that
\begin{equation*}
u_q(x, y, t) - u_*(x, y) - \sum_{j = 1}^k\left[\omega\left(\frac{x-q_i}{\lambda_i}, \frac{y}{\lambda_i}\right) - \omega(\infty)\right]\to 0 \text{ as }t\nearrow T,
\end{equation*}
in the $H^1$ and uniform senses in $\mathbb{R}^2_+$ where
\begin{equation*}
\lambda_i(t) = k^*_i\frac{T-t}{|\log(T-t)|^2}(1 + o(1)) \text{ as }t\nearrow T.
\end{equation*}
In particular, we have
\begin{equation*}
|\nabla u_q(\cdot, \cdot, t)|^2\rightharpoonup |\nabla u_*|^2 + 2\pi\sum_{j=1}^k\delta_{q_j}\text{ as }t\nearrow T.
\end{equation*}
\end{theorem}

To prove this theorem, we will use the {\it inner-outer gluing scheme} which was proved to be useful in singular perturbation elliptic problems, for example, \cite{delpinoKWeiCPAM2007}, \cite{delwei2011Degiorgi}, \cite{delkowalczykweijdg2013entire}. This method has also been developed into various parabolic flows, for example, the infinite time blowing-up solutions for critical nonlinear heat equation \cite{cortazar2016green}, \cite{del2017infinite}, singularity formation for two dimensional harmonic map flow \cite{davila2017singularity}, type II ancient solution for Yamabe flow \cite{del2012type}.

Results similar to Theorem \ref{t:main} have been established by Davila, del Pino and the second author in \cite{davila2017singularity} in the case of two dimensional harmonic map flow into ${\mathbb S}^2$, see \cite{RaphaelAS} for earlier results in the corrotational case. Comparing with \cite{davila2017singularity}, the main difficulty in this paper is the {\em nonlocality} of the problem (\ref{e:main}). In fact, according to \cite{stinga2015regularity}, we can write problem (\ref{e:main}) as
\begin{equation*}
\sqrt{\partial_t-\Delta}u = \frac{1}{8\pi}\left[\int_{0}^{+\infty}\int_{\mathbb{R}}|u(x,0,t) - u(x-z,0,t-\tau)|^2\frac{e^{-\frac{|z|^2}{4\tau}}}{\tau^2}dzd\tau\right]u(x,0,t).
\end{equation*}
The problem under consideration interpolates between the two-dimensional harmonic map flow and the half-harmonic map flow. It inherits characteristics from both problems. In \cite{SWY} we showed that for the half-harmonic map flow, infinite time blow-up exists. In this paper we combine techniques from both papers \cite{davila2017singularity} and \cite{SWY} to prove finite time blow up for (\ref{e:main}), which is unknown even in the corrotational case. The flow under consideration is actually a reminiscence of a nonlocal geometric flow involving the operator $\sqrt{\partial_t -\Delta}$ described in \cite{stinga2015regularity}, as previously mentioned, and enjoys nice monotonicity properties (see for instance \cite{BGaro} for general considerations ). The techniques used in the present paper can also used to deal with infinite-time blow up for the flow
\begin{equation*}
\label{eq:up}
\begin{cases}
u_t-\Delta u=0\hbox{ in } \mathbb{R}^{m+1}_+\times (0,\infty),\\
-\lim_{y\to 0^+}\frac{\partial u(.,t)}{\partial y}=u^{p_*},
\end{cases}
\end{equation*}
where $p_*$ is the critical exponent for the Trace Sobolev embedding. We plan to come back to this problem later.

\section{Construction of the approximate solution}
From \cite{stinga2015regularity}, we know that problem (\ref{e:main}) is equivalent to
\begin{equation}\label{e:main2.1}
\begin{cases}
u_t = \Delta u\text{ in }\mathbb{R}^2_+\times (0, T),\\
u(x,0,t) \in \mathbb{S}^1\text{ for all }(x,0,t)\in \partial\mathbb{R}^2_+\times (0, T),\\
-\frac{du}{dy}(x,0,t) = \frac{1}{8\pi}\left[\int_{0}^{+\infty}\int_{\mathbb{R}}|u(x,0,t) - u(x-z,0,t-\tau)|^2\frac{e^{-\frac{|z|^2}{4\tau}}}{\tau^2}dzd\tau\right]u(x,0,t)\\ \quad\quad\quad\quad\quad\quad\quad\quad\quad\quad\quad\quad\quad\quad\quad\quad\quad\quad\quad\quad\quad\text{ for all }(x,0,t)\in \partial\mathbb{R}^2_+\times (0, T),\\
u(x, y, t) = u_0(x, y)\text{ for all }(x, y, t)\in \mathbb{R}^2_+\times (-\infty, 0].
  \end{cases}
\end{equation}
Note that we use the factor $\frac{1}{8\pi}$ to keep (\ref{e:main2.1}) agree with the half-harmonic map equation when $u$ is independent of $t$. Here and in the following, $\frac{du}{dy}(x,0,t)$ always means $\frac{du}{dy}|_{(x,0,t)}$.
\subsection{Setting up the problem.}
Our aim is to find a solution of (\ref{e:main2.1}) which looks like
\begin{equation*}
U(x, y, t): = U_{\lambda, \xi}(x, y) = \omega\left(\frac{x-\xi}{\lambda},\frac{y}{\lambda}\right)
\end{equation*}
at main order, where
$\omega(x,y) =
\begin{pmatrix}
     \frac{2x}{x^2+(y+1)^2}\\
     \frac{x^2+y^2-1}{x^2+(y+1)^2}
\end{pmatrix}$
is the extension form of the canonical least energy half-harmonic map (\ref{e:halfharmonicmapcanonical}).
We look for parameter functions $\lambda(t)$ and $\xi(t)$ of class $C^1$ satisfying
\begin{equation*}
\lim_{t\to T}\lambda(t) = 0,\quad \lim_{t\to T}\xi(t) = q\in \partial\mathbb{R}^2_+,
\end{equation*}
and a solution to (\ref{e:main2.1}) with form $u(x, y, t) = U(x, y, t) + \varphi(x, y, t)$ blowing up at $t = T$ and
the point $(q, 0)$. Here $\varphi(x, y, t)$ is a small perturbation term.

Note that problem (\ref{e:main2.1}) is also equivalent to
\begin{equation}\label{e:halfheatequation}
\sqrt{\partial_t-\Delta}u = \frac{1}{8\pi}\left[\int_{0}^{+\infty}\int_{\mathbb{R}}|u(x,0,t) - u(x-z,0,t-\tau)|^2\frac{e^{-\frac{|z|^2}{4\tau}}}{\tau^2}dzd\tau\right]u(x,0,t).
\end{equation}
for all $(x, 0, t)\in \partial\mathbb{R}^2_+\times (-\infty, T)$. We refer the interested readers to \cite{stinga2015regularity} for the definition of $\sqrt{\partial_t-\Delta}u$. Since $u(x,0,t)\in \mathbb{S}^1$, as in \cite{davila2017singularity},
we parameterize the admissible perturbation by free small functions $\varphi:\partial\mathbb{R}^2_+\times (-\infty,T)\to \mathbb{R}^2$ with the following form
\begin{equation*}
p(\varphi):=\Pi_{U^\perp}\varphi + a(\Pi_{U^\perp}\varphi)U,
\end{equation*}
where
$$
\Pi_{U^\perp}\varphi:=\varphi-(\varphi\cdot U)U,\quad a(\Pi_{U^\perp}\varphi):=\sqrt{1+(\varphi\cdot U)^2-|\varphi|^2}-1 = \sqrt{1-|\Pi_{U^\perp}\varphi|^2}-1,
$$
hence $|U+p(\varphi)|^2=1$ holds on $\partial\mathbb{R}^2_+\times (-\infty, T)$. Considering the error operator defined as
\begin{equation*}
\begin{aligned}
S(u) &= -\sqrt{\partial_t-\Delta}u\\
 &\quad + \frac{1}{8\pi}\left[\int_{0}^{+\infty}\int_{\mathbb{R}}|u(x,0,t) - u(x-z,0,t-\tau)|^2\frac{e^{-\frac{|z|^2}{4\tau}}}{\tau^2}dzd\tau\right]u(x,0,t),
\end{aligned}
\end{equation*}
a useful observation is that if $\varphi$ solves
\begin{equation}\label{e:balanceequation}
S(U+\Pi_{U^\perp}\varphi + a(\Pi_{U^\perp})U)+\tilde{b}(x,0,t)U = 0
\end{equation}
for some scalar function $\tilde{b}(x,0,t)$ and $|\varphi| < \frac{1}{2}$, then $u=U+\Pi_{U^\perp}\varphi + a(\Pi_{U^\perp}\varphi)U$ satisfies (\ref{e:halfheatequation}), that is to say, $S(U+\Pi_{U^\perp}\varphi + a(\Pi_{U^\perp})U) = 0$. Indeed, since $|u| = 1$, $-\tilde{b}(x,t)U\cdot u = S(u)\cdot u = 0$.
On the other hand, since $|\varphi|\leq \frac{1}{2}$, $|a(\Pi_{U^\perp}\varphi)|\leq \frac{1}{4}$, thus $U\cdot u = 1 + a(\Pi_{U^\perp}\varphi) > 0$ and therefore $\tilde{b}\equiv 0$. Hence we only need to solve (\ref{e:balanceequation}). Equivalently, we will find $\varphi:\mathbb{R}^2_+\to \mathbb{R}$ such that
\begin{equation*}
\begin{cases}
&(U + \varphi)_t = \Delta (U + \varphi)\text{ in }\mathbb{R}^2_+\times (0, T),\\
&-\frac{d(U + \varphi)}{dy}(x,0,t)= \\
&\quad \left(\frac{1}{8\pi}\left[\int_{0}^{+\infty}\int_{\mathbb{R}}|u(x,0,t) - u(x-z,0,t-\tau)|^2\frac{e^{-\frac{|z|^2}{4\tau}}}{\tau^2}dzd\tau\right]u(x,0,t)\right)\Bigg|_{u = U+\varphi}\\ &\quad
\left(U + \varphi\right)+\tilde{b}(x, 0, t)U\text{ for all }(x,0,t)\in \partial\mathbb{R}^2_+\times (0, T)
\end{cases}
\end{equation*}
holds. Let us define the error operators as
\begin{eqnarray*}
S_1(u) = -u_t +\Delta u \text{ in }\mathbb{R}^2_+\times (0, T),
\end{eqnarray*}
\begin{equation*}
\begin{aligned}
S_2(u) & = \frac{du}{dy}(x,0,t)\\ &\quad +\frac{1}{8\pi}\left[\int_{0}^{+\infty}\int_{\mathbb{R}}|u(x,0,t) - u(x-z,0,t-\tau)|^2\frac{e^{-\frac{|z|^2}{4\tau}}}{\tau^2}dzd\tau\right]u(x,0,t)
\end{aligned}
\end{equation*}
in $(x,0,t)\in \partial\mathbb{R}^2_+\times (0, T)$.
For each fixed $t$, since $U$ is a half-harmonic map, we have
\begin{equation*}
\begin{cases}
\Delta U = 0\text{ in }\mathbb{R}^2_+,\\
-\frac{dU}{dy}(x, 0, t) = \frac{1}{8\pi}\left[\int_{0}^{+\infty}\int_{\mathbb{R}}|U(x,0,t) - U(x-z,0,t)|^2\frac{e^{-\frac{|z|^2}{4\tau}}}{\tau^2}dzd\tau\right]U(x,0,t) \text{ in }\partial\mathbb{R}^2_+.
  \end{cases}
\end{equation*}
Hence $S_1(U) = -U_t$ and
\begin{equation*}
\begin{aligned}
S_2(U) & = \Bigg\{\frac{1}{8\pi}\left[\int_{0}^{+\infty}\int_{\mathbb{R}}|U(x,0,t) - U(x-z,0,t-\tau)|^2\frac{e^{-\frac{|z|^2}{4\tau}}}{\tau^2}dzd\tau\right]\\
&   \quad - \frac{1}{8\pi}\left[\int_{0}^{+\infty}\int_{\mathbb{R}}|U(x,0,t) - U(x-z,0,t)|^2\frac{e^{-\frac{|z|^2}{4\tau}}}{\tau^2}dzd\tau\right]\Bigg\}U(x,0,t)
\end{aligned}
\end{equation*}

Now we compute
\begin{equation}\label{e:error}
\begin{aligned}
0 = S_1(U+\varphi) &= -U_t -\partial_t\varphi + \Delta\varphi
\end{aligned}
\end{equation}
and
\begin{equation}\label{e:error1}
\begin{aligned}
0 = S_2(U+\varphi) & = S_2(U+\Pi_{U^\perp}\varphi + aU)\\
& = \frac{d}{dy}\varphi(x, 0, \tau) + L_U(\Pi_{U^\perp}\varphi)+N_U(\Pi_{U^\perp}\varphi)+b(\Pi_{U^\perp}\varphi)U
\end{aligned}
\end{equation}
where
\begin{equation*}
\begin{aligned}
&L_U(\Pi_{U^\perp}\varphi)\\
& = \frac{1}{8\pi}\left[\int_{0}^{+\infty}\int_{\mathbb{R}}|U(x,0,t) - U(x-z,0,t-\tau)|^2\frac{e^{-\frac{|z|^2}{4\tau}}}{\tau^2}dzd\tau\right]\Pi_{U^\perp}\varphi\\
&\quad + \frac{1}{4\pi}\left[\int_{0}^{+\infty}\int_{\mathbb{R}}(U(x,0,t)-U(x-z,0,t-\tau))\right.\\
&\quad\quad\quad\quad\quad\quad\quad\left.\cdot\left(\Pi_{U^\perp}\varphi(x,0,t) -\Pi_{U^\perp}\varphi(x-z,0,t-\tau)\right)\frac{e^{-\frac{|z|^2}{4\tau}}}{\tau^2}dzd\tau\right]U(x, 0, t),
\end{aligned}
\end{equation*}
\begin{eqnarray*}
&&N_U(\Pi_{U^\perp}\varphi)\\ &&= \left(\frac{1}{4\pi}\int_{0}^{+\infty}\int_{\mathbb{R}}(a(x,0,t)U(x,0,t)-a(x-z,0,t-\tau)U(x-z,0,t-\tau))\right.\\
&&\quad\cdot(U(x,0,t)+\Pi_{U^\perp}\varphi(x,0,t)-U(x-z,0,t-\tau) -\Pi_{U^\perp}\varphi(x-z,0,t-\tau))\\
&&\quad\quad\quad\quad\quad\quad\quad\quad\quad\quad\quad\quad\quad\quad\quad\quad\quad\quad\quad\quad\quad\quad\quad\quad\quad\quad\times
\frac{e^{-\frac{|z|^2}{4\tau}}}{\tau^2}dzd\tau\\ &&\quad\left.+\frac{1}{4\pi}\int_{0}^{+\infty}\int_{\mathbb{R}}(U(x,0,t)-U(x-z,0,t-\tau))\right.\\
&&\quad\quad\quad\quad\quad\quad\quad\quad\quad\quad\quad
\left.\cdot(\Pi_{U^\perp}\varphi(x,0,t) -\Pi_{U^\perp}\varphi(x-z,0,t-\tau))\frac{e^{-\frac{|z|^2}{4\tau}}}{\tau^2}dzd\tau\right.\\
&&\quad\left. +\frac{1}{8\pi}\int_{0}^{+\infty}\int_{\mathbb{R}}(\Pi_{U^\perp}\varphi(x,0,t)-\Pi_{U^\perp}\varphi(x-z,0,t-\tau))\right.\\
&&\quad\quad\quad\quad\quad\quad\quad\quad\quad\quad\quad
\left.\cdot(\Pi_{U^\perp}\varphi(x,0,t) -\Pi_{U^\perp}\varphi(x-z,0,t-\tau))\frac{e^{-\frac{|z|^2}{4\tau}}}{\tau^2}dzd\tau\right.\\
&&\quad \left.+\frac{1}{8\pi}\int_{0}^{+\infty}\int_{\mathbb{R}}(a(x,0,t)U(x,0,t)-a(x-z,0,t-\tau)U(x-z,0,t-\tau))\right.\\
&&\quad
\left.\cdot(a(x,0,t)U(x,0,t)-a(x-z,0,t-\tau)U(x-z,0,t-\tau))\frac{e^{-\frac{|z|^2}{4\tau}}}{\tau^2}dzd\tau\right)\Pi_{U^\perp}\varphi
\end{eqnarray*}
and
\begin{equation*}
\begin{aligned}
& b(\Pi_{U^\perp}\varphi) =  \left(\frac{1}{8\pi}\int_{0}^{+\infty}\int_{\mathbb{R}}|(U+\Pi_{U^\perp}\varphi + a(\Pi_{U^\perp}\varphi)U)(x,0,t)\right. \\
&\quad \left.-(U+\Pi_{U^\perp}\varphi + a(\Pi_{U^\perp}\varphi)U)(x-z,0,t-\tau)|^2\frac{e^{-\frac{|z|^2}{4\tau}}}{\tau^2}dyd\tau\right)(1+a)\\
\quad\quad\quad &\quad - \frac{1}{4\pi}\int_{0}^{+\infty}\int_{\mathbb{R}}(U(x,0,t)-U(x-z,0,t-\tau))\\
&\quad\quad\quad\quad\quad\quad\quad\quad\quad
\cdot(\Pi_{U^\perp}\varphi(x,0,t) -\Pi_{U^\perp}\varphi(x-z,0,t-\tau))\frac{e^{-\frac{|z|^2}{4\tau}}}{\tau^2}dyd\tau\\
&\quad - \frac{1}{8\pi}\left[\int_{0}^{+\infty}\int_{\mathbb{R}}|U(x,0,t) - U(x-z,0,t)|^2\frac{e^{-\frac{|z|^2}{4\tau}}}{\tau^2}dzd\tau\right].
\end{aligned}
\end{equation*}

By direct computations, we have
\begin{eqnarray*}
S_1(U) = -U_t = -\mathcal{E}_0(x, y, t) - \mathcal{E}_1(x, y, t).
\end{eqnarray*}
Here
\begin{equation}\label{e:mathcale_0}
\mathcal{E}_0(x,y,t) = \begin{pmatrix}
     \frac{2(x-\xi)[(x-\xi)^2+(y^2-\lambda^2)]}{((x-\xi)^2+(y+\lambda)^2)^2}\\
     \frac{-2y(y+\lambda)^2 -2(x-\xi)^2(y+2\lambda)}{((x-\xi)^2+(y+\lambda)^2)^2}
\end{pmatrix}\dot{\lambda} \approx
\begin{pmatrix}
     \frac{2(x-\xi)}{(x-\xi)^2+y^2+\lambda^2}\\
     \frac{-2y}{(x-\xi)^2+y^2+\lambda^2}
\end{pmatrix}\dot{\lambda},
\end{equation}
\begin{equation*}
\mathcal{E}_1(x,y,t) = \begin{pmatrix}
     \frac{2\lambda(x+y+\lambda-\xi)(x-y-\lambda-\xi)}{((x-\xi)^2+(y+\lambda)^2)^2}\\
     \frac{-4\lambda(x-\xi)(y+\lambda)}{((x-\xi)^2+(y+\lambda)^2)^2}
\end{pmatrix}\dot{\xi}.
\end{equation*}
Since $\mathcal{E}_0(x,y,t)$ is not $L^2$ integrable, we shall decompose the correction $\varphi$ into $\varphi = \Phi^* + \Phi$ and system (\ref{e:error}), (\ref{e:error1}) transforms into the following
\begin{equation}\label{e:recast1}
\begin{aligned}
0=S_1(U+\varphi) &= -U_t -\partial_t\Phi^* + \Delta\Phi^* -\partial_t\Phi + \Delta\Phi
\end{aligned}
\end{equation}
and
\begin{equation}\label{e:recast2}
\begin{aligned}
 0 & = \frac{d}{dy}\Phi^*(x, 0, \tau) + L_U(\Pi_{U^\perp}\Phi^*)\\
&\quad + \frac{d}{dy}\Phi(x, 0, \tau) + L_U(\Pi_{U^\perp}\Phi) +N_U(\Pi_{U^\perp}(\Phi^*+\Phi))+\tilde{b}(x,0,t)U.
\end{aligned}
\end{equation}
The correction $\Phi^*$ will be chosen such that the term $-\mathcal{E}_0$ is canceled at main order away from the blow up point $(\xi, 0)$.
\subsection{The definition of $\Phi^*$.}
Let us consider the linear problem (\ref{e:recast1}),
\begin{equation*}
0 = -\partial_t\Phi + \Delta\Phi + \mathcal{E}^*_1
\end{equation*}
where
\begin{equation*}
\mathcal{E}^*_1 = -U_t -\partial_t\Phi^* + \Delta\Phi^*.
\end{equation*}
Our aim is to construct a function $\Phi^*$ such that $\mathcal{E}^*_1$ is smaller than the largest term $-\mathcal{E}_0$ of the initial error $-U_t$ given by (\ref{e:mathcale_0}) away from the blow-up point $q$.

As in \cite{davila2017singularity}, we decompose $\Phi^*$ into the following form
\begin{equation*}
\Phi^*:=\Phi^0[\lambda,\xi] + Z^*(x,y,t)
\end{equation*}
where
\begin{equation*}
Z^*(x,y,t)= \begin{pmatrix}
     z^*_1(x,y,t) \\
     z^*_2(x,y,t)
\end{pmatrix}
\end{equation*}
is a solution of the heat equation
\begin{equation*}
\left\{\begin{array}{ll}
        \partial_t Z^* = \Delta Z^*\text{ in }\mathbb{R}^2_+\times (0, T),\\
        -\frac{d}{dy}Z^*(x,0,t) = 0\text{ in }\partial\mathbb{R}^2_+\times (0, T),\\
        Z^*(x, y, 0) = Z^*_0(x, y)
       \end{array}
\right.
\end{equation*}
independent of the parameter functions. Further assumptions on $Z^*_0$ will be given in subsection 2.5. $\Phi^0[\lambda,\xi]$ is an explicit function satisfying
\begin{equation}\label{e:approximationphi}
-\mathcal{E}_0 -\partial_t\Phi^* + \Delta\Phi^*\approx 0.
\end{equation}
Observe that if $\phi^0$ is a solution to
$$-\phi^0_t + \Delta \phi^0  - \begin{pmatrix}
     \frac{2(x-\xi)}{(x-\xi)^2+y^2+\lambda^2}\\
     \frac{-2y}{(x-\xi)^2+y^2+\lambda^2}
\end{pmatrix}\dot{\lambda} = 0,$$
then $\Phi^0 = \phi^0$ will satisfy (\ref{e:approximationphi}). Set
$$
p(t) = -2\dot{\lambda},
$$
$$
z(r) = \sqrt{(x-\xi)^2 + y^2 + \lambda^2},
$$
$$
\phi^0(x, y, t) = \begin{pmatrix}x-\xi\\-y\end{pmatrix}\psi(z(r),t).
$$
Then $\psi(z, t)$ satisfies
$$
\psi_t = \psi_{zz} + \frac{3\psi_z}{z} + \frac{p(t)}{z^2}
$$
which is the radially symmetric form of an inhomogeneous heat equation in $\mathbb{R}^4$. Then Duhamel's formula gives the following expression for a weak solution
$$
\psi(z, t) = \int_{-T}^tp(s)k(z, t-s)ds, \quad k(z, t) = \frac{1-e^{-\frac{z^2}{4t}}}{z^2},
$$
where $p(t)$ is also defined for negative values of $t$ by setting $p(t) = -2\dot{\lambda}(0)$ for $t\in [-T, 0)$. Now we define
$$
\Phi^0 = \begin{pmatrix}
     \varphi^0\\
     \varphi^1
\end{pmatrix},
$$
$$
\varphi^0 = (x-\xi)\int_{-T}^tp(s)k(z(r), t-s)ds
$$
and
$$
\varphi^1 = -y\int_{-T}^tp(s)k(z(r), t-s)ds.
$$
Now, we compute
\begin{equation*}
-\Phi^0_t = \begin{pmatrix}
\begin{aligned} - (x-\xi)\psi_t & + \dot{\xi}\int_{-T}^tp(s)k(z(r), t-s)ds\\ &+ \frac{x-\xi}{z^2}\int_{-T}^tp(s)zk_z(z(r), t-s)ds\left[(x-\xi)\dot{\xi} - \dot{\lambda}\lambda\right]\end{aligned}\\
y\psi_t - \frac{y}{z^2}\int_{-T}^tp(s)zk_z(z(r), t-s)ds\left[(x-\xi)\dot{\xi} - \dot{\lambda}\lambda\right]
\end{pmatrix},
\end{equation*}
\begin{equation*}
\Delta \Phi^0 =
\begin{pmatrix}
\begin{aligned}3(x-\xi)\int_{-T}^tp(s)k_z\frac{1}{z}ds &+ (x-\xi)\int_{-T}^tp(s)k_{zz}ds\\ &+ (x-\xi)\frac{\lambda^2}{z^4}\int_{-T}^tp(s)[k_zz-k_{zz}z^2]ds\end{aligned}\\
-3y\int_{-T}^tp(s)k_z\frac{1}{z}ds -y\int_{-T}^tp(s)k_{zz}ds -y\frac{\lambda^2}{z^4}\int_{-T}^tp(s)[k_zz-k_{zz}z^2]ds
\end{pmatrix}.
\end{equation*}
Therefore, we have
$$
-\Phi^0_t + \Delta \Phi^0 = \tilde{\mathcal{R}}_0 + \tilde{\mathcal{R}}_1,
$$
where
\begin{equation*}
\tilde{\mathcal{R}}_0 = \begin{pmatrix}-(x-\xi)\frac{p(t)}{z^2} + \frac{(x-\xi)\lambda^2}{z^4}\int_{-T}^tp(s)[zk_z - z^2k_{zz}]ds\\
y\frac{p(t)}{z^2} - \frac{y\lambda^2}{z^4}\int_{-T}^tp(s)[zk_z - z^2k_{zz}]ds
\end{pmatrix},
\end{equation*}
\begin{equation*}
\tilde{\mathcal{R}}_1 = \begin{pmatrix}\dot{\xi}\int_{-T}^tp(s)kds + \frac{(x-\xi)}{z^2}\left[(x-\xi)\dot{\xi} - \dot{\lambda}\lambda\right]\int_{-T}^tp(s)zk_zds\\
-\frac{y}{z^2}\left[(x-\xi)\dot{\xi} - \dot{\lambda}\lambda\right]\int_{-T}^tp(s)zk_zds
\end{pmatrix}.
\end{equation*}
\subsection{Estimate of the inner error.}
Now we compute the inner error $\mathcal{E}^*_1 := -\Phi^*_t + \Delta\Phi^* - U_t$ as
\begin{equation*}
\begin{aligned}
\mathcal{E}^*_1 &= -\Phi^*_t + \Delta\Phi^* - U_t\\
& =  -[\Phi^0 + Z^*]_t + \Delta [\Phi^0 + Z^*] - U_t\\
& = \tilde{\mathcal{R}}_0 + \tilde{\mathcal{R}}_1 - U_t\\
& =  \begin{pmatrix}-(x-\xi)\frac{p(t)}{z^2} + \frac{(x-\xi)\lambda^2}{z^4}\int_{-T}^tp(s)[zk_z - z^2k_{zz}]ds\\
y\frac{p(t)}{z^2} - \frac{y\lambda^2}{z^4}\int_{-T}^tp(s)[zk_z - z^2k_{zz}]ds
\end{pmatrix}\\ &\quad + \begin{pmatrix}\dot{\xi}\int_{-T}^tp(s)kds + \frac{(x-\xi)}{z^2}\left[(x-\xi)\dot{\xi} - \dot{\lambda}\lambda\right]\int_{-T}^tp(s)zk_zds\\
-\frac{y}{z^2}\left[(x-\xi)\dot{\xi} - \dot{\lambda}\lambda\right]\int_{-T}^tp(s)zk_zds
\end{pmatrix} \\
&\quad - \begin{pmatrix}
     \frac{2(x-\xi)[(x-\xi)^2+(y^2-\lambda^2)]}{((x-\xi)^2+(y+\lambda)^2)^2}\\
     \frac{-2y(y+\lambda)^2 -2(x-\xi)^2(y+2\lambda)}{((x-\xi)^2+(y+\lambda)^2)^2}
\end{pmatrix}\dot{\lambda}  - \begin{pmatrix}
     \frac{2\lambda(x+y+\lambda-\xi)(x-y-\lambda-\xi)}{((x-\xi)^2+(y+\lambda)^2)^2}\\
     \frac{-4\lambda(x-\xi)(y+\lambda)}{((x-\xi)^2+(y+\lambda)^2)^2}
\end{pmatrix}\dot{\xi}\\
& =  \begin{pmatrix}2(x-\xi)\frac{\dot{\lambda}}{z^2} + \frac{(x-\xi)\lambda^2}{z^4}\int_{-T}^tp(s)[zk_z - z^2k_{zz}]ds\\
-2y\frac{\dot{\lambda}}{z^2} - \frac{y\lambda^2}{z^4}\int_{-T}^tp(s)[zk_z - z^2k_{zz}]ds
\end{pmatrix}\\ &\quad + \begin{pmatrix}\dot{\xi}\int_{-T}^tp(s)kds + \frac{(x-\xi)}{z^2}\left[(x-\xi)\dot{\xi} - \dot{\lambda}\lambda\right]\int_{-T}^tp(s)zk_zds\\
-\frac{y}{z^2}\left[(x-\xi)\dot{\xi} - \dot{\lambda}\lambda\right]\int_{-T}^tp(s)zk_zds
\end{pmatrix} \\
&\quad - \begin{pmatrix}
     \frac{2(x-\xi)[(x-\xi)^2+(y^2-\lambda^2)]}{((x-\xi)^2+(y+\lambda)^2)^2}\\
     \frac{-2y(y+\lambda)^2 -2(x-\xi)^2(y+2\lambda)}{((x-\xi)^2+(y+\lambda)^2)^2}
\end{pmatrix}\dot{\lambda}  - \begin{pmatrix}
     \frac{2\lambda(x+y+\lambda-\xi)(x-y-\lambda-\xi)}{((x-\xi)^2+(y+\lambda)^2)^2}\\
     \frac{-4\lambda(x-\xi)(y+\lambda)}{((x-\xi)^2+(y+\lambda)^2)^2}
\end{pmatrix}\dot{\xi},
\end{aligned}
\end{equation*}
hence
\begin{equation*}
\begin{aligned}
\mathcal{E}^*_1
& = \dot{\lambda}\begin{pmatrix}
     \frac{2(x-\xi)}{z^2} -\frac{2(x-\xi)[(x-\xi)^2+(y^2-\lambda^2)]}{((x-\xi)^2+(y+\lambda)^2)^2}\\
     \frac{-2y}{z^2} + \frac{2y(y+\lambda)^2 +2(x-\xi)^2(y+2\lambda)}{((x-\xi)^2+(y+\lambda)^2)^2}
\end{pmatrix} - \dot{\xi}\begin{pmatrix}
     \frac{2\lambda(x+y+\lambda-\xi)(x-y-\lambda-\xi)}{((x-\xi)^2+(y+\lambda)^2)^2}\\
     \frac{-4\lambda(x-\xi)(y+\lambda)}{((x-\xi)^2+(y+\lambda)^2)^2}
\end{pmatrix}\\
&\quad +\frac{\lambda^2}{z^4}\int_{-T}^tp(s)[zk_z - z^2k_{zz}]ds\begin{pmatrix}x-\xi\\
- y
\end{pmatrix}\\
&\quad + \frac{\left[(x-\xi)\dot{\xi} - \dot{\lambda}\lambda\right]}{z^2}\int_{-T}^tp(s)zk_zds\begin{pmatrix}x-\xi\\
- y
\end{pmatrix} + \int_{-T}^tp(s)kds \begin{pmatrix}\dot{\xi}\\
0
\end{pmatrix}\\
& =  \dot{\lambda}\begin{pmatrix}
     \frac{2(x-\xi)}{z^2} -\frac{2(x-\xi)[(x-\xi)^2+(y^2-\lambda^2)]}{((x-\xi)^2+(y+\lambda)^2)^2}\\
     \frac{-2y}{z^2} + \frac{2y(y+\lambda)^2 +2(x-\xi)^2(y+2\lambda)}{((x-\xi)^2+(y+\lambda)^2)^2}
\end{pmatrix} - \dot{\xi}\begin{pmatrix}
     \frac{2\lambda(x+y+\lambda-\xi)(x-y-\lambda-\xi)}{((x-\xi)^2+(y+\lambda)^2)^2}\\
     \frac{-4\lambda(x-\xi)(y+\lambda)}{((x-\xi)^2+(y+\lambda)^2)^2}
\end{pmatrix}\\
&\quad + \frac{\rho}{\lambda(\rho^2+1)^2}\int_{-T}^tp(s)[zk_z - z^2k_{zz}]ds\begin{pmatrix}\frac{x-\xi}{r}\\
- \frac{y}{r}
\end{pmatrix}\\
&\quad + \frac{\left[(x-\xi)\dot{\xi} - \dot{\lambda}\lambda\right]r}{z^2}\int_{-T}^tp(s)zk_zds\begin{pmatrix}\frac{x-\xi}{r}\\
- \frac{y}{r}
\end{pmatrix} + \int_{-T}^tp(s)kds \begin{pmatrix}\dot{\xi}\\
0
\end{pmatrix}.
\end{aligned}
\end{equation*}
Here we have used the notations $r = \sqrt{(x-\xi)^2 + y^2}$, $\rho = r/\lambda$ and $r\lambda^2/z^4 = \frac{\rho}{\lambda(\rho^2+1)^2}$.
Furthermore, we have
\begin{equation*}
\begin{aligned}
\mathcal{E}^*_1
& = \dot{\lambda}\begin{pmatrix}
     \frac{2(x-\xi)}{z^2} -\frac{2(x-\xi)[(x-\xi)^2+(y^2-\lambda^2)]}{((x-\xi)^2+(y+\lambda)^2)^2}\\
     \frac{-2y}{z^2} + \frac{2y(y+\lambda)^2 +2(x-\xi)^2(y+2\lambda)}{((x-\xi)^2+(y+\lambda)^2)^2}
\end{pmatrix} - \dot{\xi}\begin{pmatrix}
     \frac{2\lambda(x+y+\lambda-\xi)(x-y-\lambda-\xi)}{((x-\xi)^2+(y+\lambda)^2)^2}\\
     \frac{-4\lambda(x-\xi)(y+\lambda)}{((x-\xi)^2+(y+\lambda)^2)^2}
\end{pmatrix}\\
&\quad + \frac{\rho}{\lambda(\rho^2+1)^2}\int_{-T}^tp(s)[zk_z - z^2k_{zz}]ds\begin{pmatrix}\frac{x-\xi}{r}\\
- \frac{y}{r}
\end{pmatrix}\\
&\quad + \frac{\left[(x-\xi)\dot{\xi} - \dot{\lambda}\lambda\right]r}{z^2}\int_{-T}^tp(s)zk_zds\begin{pmatrix}\frac{x-\xi}{r}\\
- \frac{y}{r}
\end{pmatrix}+ \int_{-T}^tp(s)kds \begin{pmatrix}\dot{\xi}\\
0
\end{pmatrix}\\
& =  \dot{\lambda}\begin{pmatrix}\frac{2(x-\xi)}{r}\\
- \frac{2y}{r}
\end{pmatrix}r\left[\frac{1}{(x-\xi)^2+y^2+\lambda^2}- \frac{(x-\xi)^2+y^2}{\left[(x-\xi)^2+y^2+\lambda^2\right]^2}\right]\\
&\quad + \dot{\lambda}\begin{pmatrix}
     2(x-\xi)\frac{(x-\xi)^2+y^2}{\left[(x-\xi)^2+y^2+\lambda^2\right]^2} -\frac{2(x-\xi)[(x-\xi)^2+(y^2-\lambda^2)]}{((x-\xi)^2+(y+\lambda)^2)^2}\\
     -2y\frac{(x-\xi)^2+y^2}{\left[(x-\xi)^2+y^2+\lambda^2\right]^2} + \frac{2y(y+\lambda)^2 +2(x-\xi)^2(y+2\lambda)}{((x-\xi)^2+(y+\lambda)^2)^2}
\end{pmatrix}\\
&\quad - \dot{\xi}\begin{pmatrix}
     \frac{2\lambda(x+y+\lambda-\xi)(x-y-\lambda-\xi)}{((x-\xi)^2+(y+\lambda)^2)^2}\\
     \frac{-4\lambda(x-\xi)(y+\lambda)}{((x-\xi)^2+(y+\lambda)^2)^2}
\end{pmatrix}\\
&\quad +\frac{\rho}{\lambda(\rho^2+1)^2}\int_{-T}^tp(s)[zk_z - z^2k_{zz}]ds\begin{pmatrix}\frac{x-\xi}{r}\\
- \frac{y}{r}
\end{pmatrix}\\
&\quad + \frac{\left[(x-\xi)\dot{\xi} - \dot{\lambda}\lambda\right]r}{z^2}\int_{-T}^tp(s)zk_zds\begin{pmatrix}\frac{x-\xi}{r}\\
- \frac{y}{r}
\end{pmatrix} + \int_{-T}^tp(s)kds \begin{pmatrix}\dot{\xi}\\
0
\end{pmatrix}\\
& =  \dot{\lambda}\frac{\rho}{\lambda(\rho^2+1)^2}\begin{pmatrix}\frac{2(x-\xi)}{r}\\
- \frac{2y}{r}
\end{pmatrix}\\
&\quad + \dot{\lambda}\begin{pmatrix}
     2(x-\xi)\frac{(x-\xi)^2+y^2}{\left[(x-\xi)^2+y^2+\lambda^2\right]^2} -\frac{2(x-\xi)[(x-\xi)^2+(y^2-\lambda^2)]}{((x-\xi)^2+(y+\lambda)^2)^2}\\
     -2y\frac{(x-\xi)^2+y^2}{\left[(x-\xi)^2+y^2+\lambda^2\right]^2} + \frac{2y(y+\lambda)^2 +2(x-\xi)^2(y+2\lambda)}{((x-\xi)^2+(y+\lambda)^2)^2}
\end{pmatrix}\\
&\quad - \dot{\xi}\begin{pmatrix}
     \frac{2\lambda(x+y+\lambda-\xi)(x-y-\lambda-\xi)}{((x-\xi)^2+(y+\lambda)^2)^2}\\
     \frac{-4\lambda(x-\xi)(y+\lambda)}{((x-\xi)^2+(y+\lambda)^2)^2}
\end{pmatrix}\\
&\quad + \frac{\rho}{\lambda(\rho^2+1)^2}\int_{-T}^tp(s)[zk_z - z^2k_{zz}]ds\begin{pmatrix}\frac{x-\xi}{r}\\
- \frac{y}{r}
\end{pmatrix}\\
&\quad + \frac{\left[(x-\xi)\dot{\xi} - \dot{\lambda}\lambda\right]r}{z^2}\int_{-T}^tp(s)zk_zds\begin{pmatrix}\frac{x-\xi}{r}\\
- \frac{y}{r}
\end{pmatrix} + \int_{-T}^tp(s)kds \begin{pmatrix}\dot{\xi}\\
0
\end{pmatrix}.
\end{aligned}
\end{equation*}

\subsection{Estimate of the boundary error.}
Equation (\ref{e:recast2}) can be approximated by the following linear problem
\begin{equation*}
0 = \mathcal{E}^*_2 + \frac{d}{dy}\Phi(x, 0, \tau) + L_U(\Pi_{U^\perp}\Phi) +N_U(\Pi_{U^\perp}(\Phi^*+\Phi))+\tilde{b}(x,0,t)U,
\end{equation*}
where
\begin{equation*}
\mathcal{E}^*_2 = \frac{d}{dy}\Phi^*(x, 0, \tau) + L_U(\Pi_{U^\perp}\Phi^*).
\end{equation*}
Now we compute the boundary error $\mathcal{E}^*_2$ with $\Phi^* = \Phi^0 + Z^*$.
First, we have
\begin{equation*}
\begin{aligned}
&\frac{d}{dy}\left[\Phi^0 + Z^*\right] = \frac{d}{dy}\Phi^0\\
& = \frac{d}{dy}\begin{pmatrix}
     (x-\xi)\int_{-T}^tp(s)k(z(r), t-s)ds\\
     -y\int_{-T}^tp(s)k(z(r), t-s)ds
\end{pmatrix}\\
& = \begin{pmatrix}
     0\\
     -\int_{-T}^tp(s)k(z(r), t-s)|_{y=0}ds
\end{pmatrix}\\
& = \begin{pmatrix}
     0\\
     -\int_{-T}^tp(s)k(\sqrt{(x-\xi)^2 + \lambda^2}, t-s)ds
\end{pmatrix}.
\end{aligned}
\end{equation*}
Then, when $T > 0$ is sufficiently small, there holds
\begin{equation*}
\begin{aligned}
\mathcal{E}^*_2 &= \frac{d}{dy}\Phi^0(x,0,t)\\
&\quad  + \frac{1}{8\pi}\left[\int_{0}^{+\infty}\int_{\mathbb{R}}|U(x,0,t) - U(x-z,0,t-\tau)|^2\frac{e^{-\frac{|z|^2}{4\tau}}}{\tau^2}dzd\tau\right]\Pi_{U^\perp}[\Phi^*]\\
&\quad  + \frac{1}{4\pi}\int_{0}^{+\infty}\int_{\mathbb{R}}(U(x,0,t)-U(x-z,0,t-\tau))\\
&\quad\quad\quad \cdot (\Pi_{U^\perp}[\Phi^*](x,0,t)-\Pi_{U^\perp}[\Phi^*](x-z,0,t-\tau))\frac{e^{-\frac{|z|^2}{4\tau}}}{\tau^2}dzd\tau U\left(x, 0, t\right)\\
& \approx  \begin{pmatrix}
     0\\
     -\int_{-T}^tp(s)k(\sqrt{(x-\xi)^2 + \lambda^2}, t-s)ds
\end{pmatrix}\Bigg|_{y = \frac{x-\xi}{\lambda}}\\
&\quad  +
\frac{2}{1+y^2}\begin{pmatrix}
y\int_{-T}^tp(s)k(\sqrt{\lambda^2 y^2 + \lambda^2}, t-s)ds\\
0
\end{pmatrix}\Bigg|_{y = \frac{x-\xi}{\lambda}}\\
&\quad  +
\frac{2}{1+y^2}\frac{1}{\lambda}\begin{pmatrix}
z_1^*(\xi+\lambda y,0,t)\\
z_2^*(\xi+\lambda y,0,t)
\end{pmatrix}\Bigg|_{y = \frac{x-\xi}{\lambda}} + b(x, 0, t)U(x, 0, t)
\end{aligned}
\end{equation*}
for some scalar function $b(x, 0, t)$ which depends on $\varphi(x, 0, t)$.

\subsection{Improve error near the blow up point: choice of $\lambda$ and $\xi$.}
System (\ref{e:recast1}) and (\ref{e:recast2}) can be approximated by the following linear problem
\begin{equation}\label{e:linearequation1}
\begin{aligned}
0= -\partial_t\varphi + \Delta\varphi + \mathcal{E}^*_1
\end{aligned}
\end{equation}
and
\begin{equation}\label{e:linearequation2}
\begin{aligned}
0 = \frac{d}{dy}\varphi(x,0,t) + L_U(\varphi) + \mathcal{E}^*_2 + b(x,0,t)U,\quad \varphi(x, 0, t)\cdot U(x, 0, t) = 0.
\end{aligned}
\end{equation}
A choice of the parameter functions is possible when suitable conditions $Z^*_0(x,y)$ are assumed.
For a point $(q, 0)\in\partial\mathbb{R}^2_+$ and a smooth function
\begin{equation*}
\tilde{Z}_0(x,y) = \begin{pmatrix}
\tilde{z}_{01}(x,y)\\ \tilde{z}_{02}(x,y)
\end{pmatrix}
\end{equation*}
satisfying
\begin{equation*}
\tilde{Z}_0(q, 0) = \begin{pmatrix}0\\ 0\end{pmatrix},\quad \partial_x \tilde{z}_{02}(q, 0) < 0,
\end{equation*}
we define
\begin{equation*}
Z_0^*:=\delta \tilde{Z}_0(x,y) = \begin{pmatrix}
z^*_{01}(x,y)\\ z^*_{02}(x,y)
\end{pmatrix}
\end{equation*}
for a fixed but small number $\delta > 0$.

If we write
$$
\varphi(x,y,t) = \phi(u, v, t),\quad u = \frac{x-\xi}{\lambda},\quad v = \frac{y}{\lambda},
$$
then (\ref{e:linearequation1}) and (\ref{e:linearequation2}) becomes
\begin{equation*}
\begin{aligned}
0= -\lambda^2\partial_t\phi + \Delta\phi + \lambda^2\mathcal{E}^*_1
\end{aligned}
\end{equation*}
and
\begin{equation*}
\begin{aligned}
0 = \frac{d}{dv}\phi(u,0,t) + L_\omega(\phi) + \lambda\mathcal{E}^*_2 + b(u, 0, t)\omega,\quad \phi\cdot \omega = 0.
\end{aligned}
\end{equation*}
Then an improvement of the approximation can be achieved if the following time-independent problem
\begin{equation}\label{e:linearequation11}
\begin{aligned}
0= \Delta\phi + \lambda^2\mathcal{E}^*_1,
\end{aligned}
\end{equation}
\begin{equation}\label{e:linearequation12}
\begin{aligned}
0 = \frac{d}{dv}\phi(u,0) + L_\omega(\phi) + \lambda\mathcal{E}^*_2, \quad\phi\cdot \omega = 0
\end{aligned}
\end{equation}
and
\begin{equation}\label{e:linearequation13}
\lim_{|(u, v)|\to \infty}\phi(u, v) = 0\text{ in }\mathbb{R}^2_+
\end{equation}
is satisfied approximately. Note that the decay condition (\ref{e:linearequation13}) is needed to not essentially modify the size of error far away from $(q, 0)$.

\subsubsection{Nondegeneracy of the half harmonic maps.} It was proved in \cite{sire2017nondegeneracy} that $\omega$ is nondegenerate, which is a crucial ingredient in the singularity formation problem of half-harmonic map flow (\cite{SWY}). Observe that $\omega$ is invariant under dilation, translation and rotation, equivalently, for $Q=\begin{pmatrix}
     \cos\alpha & -\sin\alpha \\
     \sin\alpha & \cos\alpha
\end{pmatrix}\in O(2)$, $q\in \mathbb{R}$ and $\lambda\in\mathbb{R}^+$, the function
\begin{equation*}
Q\omega\left(\frac{x-q}{\lambda}\right) = \begin{pmatrix}
     \cos\alpha & -\sin\alpha \\
     \sin\alpha & \cos\alpha
\end{pmatrix}\omega\left(\frac{x-q}{\lambda}\right)
\end{equation*}
is still a solution of problem (\ref{e:halfharmonicmap}). Differentiating with $\alpha$, $q$ and $\lambda$ respectively, then we set $\alpha = 0$, $q=0$, $\lambda = 1$ and obtain that the following three functions
\begin{equation*}
Z_1(x) =
\begin{pmatrix}
     \frac{1-x^2}{x^2+1}\\
     \frac{2x}{x^2+1}
\end{pmatrix},\quad
Z_2(x) = \begin{pmatrix}
     \frac{2(x^2-1)}{(x^2+1)^2}\\
     \frac{-4x}{(x^2+1)^2}
\end{pmatrix},\quad
Z_3(x) =
\begin{pmatrix}
     \frac{2x(x^2-1)}{(x^2+1)^2}\\
     \frac{-4x^2}{(x^2+1)^2},
\end{pmatrix}
\end{equation*}
which satisfy the linearized equation at $\omega$ of (\ref{e:halfharmonicmap}) defined by
\begin{eqnarray*}
\nonumber (-\Delta)^{\frac{1}{2}}v(x) &=& \left(\frac{1}{2\pi}\int_{\mathbb{R}}\frac{|\omega(x)-\omega(y)|^2}{|x-y|^2}dy\right)v(x)\\
\quad\quad\quad &&+ \left(\frac{1}{\pi}\int_{\mathbb{R}}\frac{(\omega(x)-\omega(y))\cdot(v(x) -v(y))}{|x-y|^2}dy\right)\omega(x)\quad\text{in }\mathbb{R}
\end{eqnarray*}
for $v:\mathbb{R}\to T_U\mathbb{S}^1$.
Using this harmonic extension (see \cite{caffarelliSilvestre2007} for generalization), we have the following extension form of $\omega$ and $Z_1(x)$, $Z_2(x)$, $Z_3(x)$,
\begin{equation*}
\omega(x) =
\begin{pmatrix}
     \frac{2x}{x^2+1}\\
     \frac{x^2-1}{x^2+1}
\end{pmatrix}\to
\omega(x,y) = \begin{pmatrix}
     \frac{2x}{x^2+(y+1)^2}\\
     \frac{x^2+y^2-1}{x^2+(y+1)^2}
\end{pmatrix},
\end{equation*}
\begin{equation*}
Z_1(x) =
\begin{pmatrix}
     \frac{1-x^2}{x^2+1}\\
     \frac{2x}{x^2+1}
\end{pmatrix}\to Z_1(x, y) =
\begin{pmatrix}
\frac{1-x^2-y^2}{x^2+(y+1)^2}\\
     \frac{2x}{x^2+(y+1)^2}
\end{pmatrix},
\end{equation*}
\begin{equation*}
Z_2(x) = \begin{pmatrix}
     \frac{2(x^2-1)}{(x^2+1)^2}\\
     \frac{-4x}{(x^2+1)^2}
\end{pmatrix}\to Z_2(x, y) =
\begin{pmatrix}
     \frac{2x^2-2(y+1)^2}{(x^2+(y+1)^2)^2}\\
     \frac{-4x(y+1)}{(x^2+(y+1)^2)^2}
\end{pmatrix},
\end{equation*}
\begin{equation*}
Z_3(x) =
\begin{pmatrix}
     \frac{2x(x^2-1)}{(x^2+1)^2}\\
     \frac{-4x^2}{(x^2+1)^2}
\end{pmatrix}\to Z_3(x, y) =
\begin{pmatrix}
     \frac{2x(x^2+y^2-1)}{(x^2+(y+1)^2)^2}\\
     -\frac{2(y(y+1)^2+x^2(2+y))}{(x^2+(y+1)^2)^2}
\end{pmatrix}.
\end{equation*}

\subsubsection{Choice of $\lambda$.}
Testing (\ref{e:linearequation11}) with $Z_3(x,y)$ and integrating by parts, by the Stokes theorem and decay assumption (\ref{e:linearequation13}), it holds that
\begin{equation}\label{balancecondition1}
\lambda\int_{\mathbb{R}^2_+}\mathcal{E}^*_1\cdot Z_3dudv + \int_{\mathbb{R}}\mathcal{E}^*_2\cdot Z_3du = 0.
\end{equation}
From the computation of Section 2.3, we have
\begin{equation*}
\begin{aligned}
\lambda\int_{\mathbb{R}^2_+}\mathcal{E}^*_1\cdot Z_3 &\approx \pi\dot{\lambda} + \pi \int_{-T}^tp(s)\Gamma\left(\frac{\lambda^2}{t-s}\right)\frac{ds}{t-s}
\end{aligned}
\end{equation*}
where
\begin{equation*}
\Gamma(\tau) = \int_{0}^\infty\frac{2\rho^2}{(\rho^2+1)^2}[\zeta K_\zeta - \zeta^2K_\zeta^2]|_{\zeta = \tau(1+\rho^2)}d\rho.
\end{equation*}
On the other hand, from Section 2.4, we have
\begin{equation*}
\begin{aligned}
\int_{\mathbb{R}}\mathcal{E}^*_2\cdot Z_3 &= \int_{\mathbb{R}}\left[\frac{4x^2}{(x^2+1)^2}\int_{-T}^tp(s)k(\sqrt{\lambda^2x^2 + \lambda^2}, t-s)ds\right]dx + 2\pi b_2\\
& = \int_{-T}^t\left[p(s)\int_{\mathbb{R}}\frac{4x^2}{(x^2+1)^2} k(\sqrt{\lambda^2x^2 + \lambda^2}, t-s)dx\right]ds + 2\pi b_2\\
& = \int_{-T}^tp(s)\Gamma_b\left(\frac{\lambda(t)^2}{t-s}\right)\frac{ds}{t-s} + 2\pi b_2
\end{aligned}
\end{equation*}
where
\begin{eqnarray*}
\Gamma_b(\tau) = \int_{0}^\infty\frac{8\rho^2}{(\rho^2+1)^2}\frac{1-e^{-\frac{\zeta}{4}}}{\zeta}|_{\zeta = \tau(1+\rho^2)}d\rho
\end{eqnarray*}
and $b_2 = \partial_x z^*_{02}|_{(q, 0)}$.
Then (\ref{balancecondition1}) becomes
\begin{equation*}
\dot{\lambda} + \int_{-T}^tp(s)\Gamma\left(\frac{\lambda^2}{t-s}\right)\frac{ds}{t-s} + \frac{1}{\pi}\int_{-T}^tp(s)\Gamma_b\left(\frac{\lambda(t)^2}{t-s}\right)\frac{ds}{t-s} + 2 b_2 = 0.
\end{equation*}
Hence
\begin{equation*}
\dot{\lambda} + \int_{-T}^tp(s)\Gamma_0\left(\frac{\lambda^2}{t-s}\right)\frac{ds}{t-s} + 2 b_2 = 0,
\end{equation*}
where $\Gamma_0(\tau) = \Gamma(\tau) + \frac{1}{\pi}\Gamma_b(\tau)$. This function satisfies
\begin{equation*}
\Gamma_0(0) = c \neq 0,\quad \Gamma_0(\tau) = O(\frac{1}{\tau}) \text{ as }\tau\to +\infty.
\end{equation*}
Denote
\begin{equation*}
\mathcal{A}[\lambda,\xi] = \dot{\lambda} + \int_{-T}^tp(s)\Gamma_0\left(\frac{\lambda^2}{t-s}\right)\frac{ds}{t-s} + 2 b_2.
\end{equation*}
Now we claim that by the simple ansatz
\begin{equation*}
\dot{\lambda}(t) = -\frac{\kappa |\log T|}{\log^2(T-t)}
\end{equation*}
for some constant $\kappa > 0$, then
\begin{equation}\label{e:approximatequation}
\mathcal{A}(\lambda)(t) = o(1)
\end{equation}
will be achieved, here $o(1)$ vanishes at $t = T$ and is uniformly small with $T$. Denote
\begin{equation*}
\mathcal{B}[\lambda,\xi](t): = \int_{-T}^t\dot{\lambda}(s)\Gamma_0\left(\frac{\lambda^2}{t-s}\right)\frac{ds}{t-s}.
\end{equation*}
Similar arguments as \cite{davila2017singularity} show that
\begin{equation*}
\left|\mathcal{B}[\lambda,\xi](t) - c\kappa\right|\lesssim \kappa\frac{\log(|\log T|)}{|\log T|}.
\end{equation*}
Therefore
\begin{equation*}
\mathcal{A}[\lambda,\xi] = c\kappa(1+o(1)) + 2 b_2.
\end{equation*}
Then we assume that $\frac{d}{dx}z^*_{02}(q, 0) < 0$, (\ref{e:approximatequation}) is satisfied by choosing
\begin{equation*}
\kappa_0 = -\frac{2}{c}\frac{d}{dx}z^*_{02}(q, 0).
\end{equation*}
Define
\begin{equation}\label{e:blowuprate}
\dot{\lambda}_0(t) = -\frac{\kappa_0 |\log T|}{\log^2(T-t)}.
\end{equation}

\subsubsection{Choice of $\xi$.}
Similarly, testing (\ref{e:linearequation11}) with $Z_2(x,y)$ we get
\begin{equation}\label{balanceofxi}
\lambda\int_{\mathbb{R}^2_+}\mathcal{E}^*_1\cdot Z_2dy = \int_{\mathbb{R}}\mathcal{E}^*_2\cdot Z_2dy.
\end{equation}
By direct computations, we have
\begin{equation*}
\begin{aligned}
\lambda\int_{\mathbb{R}^2_+}\mathcal{E}^*_1\cdot Z_2 &\approx -\dot{\xi}\int_{\mathbb{R}^2_+}Z_2(u,v)\cdot Z_2(u,v)dudv = -\pi\dot{\xi}
\end{aligned}
\end{equation*}
and
\begin{equation*}
\begin{aligned}
\int_{\mathbb{R}}\mathcal{E}^*_2\cdot Z_2 &= 0.
\end{aligned}
\end{equation*}
Therefore (\ref{balanceofxi}) becomes
\begin{equation*}
-\pi\dot{\xi} \thickapprox 0.
\end{equation*}
This can be achieved by simply choosing
\begin{equation}\label{e:definitionofxi0}
\xi_0(t) = (q, 0).
\end{equation}
\subsection{The final ansatz.}
Fix $\lambda_0(t)$ defined in (\ref{e:blowuprate}) and $\xi_0(t)$ in (\ref{e:definitionofxi0}). We write
$$
\lambda(t) = \lambda_0(t) + \lambda_1(t),\quad \xi(t) = \xi_0(t) +\xi_1(t).
$$
We are looking for a small solution $\varphi$ of
\begin{equation}\label{e:equationforvarphi}
\begin{aligned}
0=\mathcal{E}^*_1 -\partial_t\varphi + \Delta\varphi
\end{aligned}
\end{equation}
and
\begin{equation}\label{e:equationforvarphi1}
\begin{aligned}
& 0 = \mathcal{E}^*_2 + \frac{d}{dy}\varphi(x, 0, t) + L_U(\Pi_{U^\perp}\varphi) + N_U(\Pi_{U^\perp}(\Phi^*+\varphi))+b(x, 0, t)U.
\end{aligned}
\end{equation}
where
\begin{equation*}
\Phi^* = \Phi[\lambda, \xi] + Z^*.
\end{equation*}
In terms of problem (\ref{e:main}), we let
$$
u = U + \Phi^*+\varphi
$$
solves the problem
\begin{equation*}
\begin{cases}
u_t = \Delta u\text{ in }\mathbb{R}^2_+\times (0, T),\\
u(x,0,t) \in \mathbb{S}^1\text{ for all }(x,0,t)\in \partial\mathbb{R}^2_+\times (0, T),\\
-\frac{du}{dy}(x,0,t) = \frac{1}{8\pi}\left[\int_{0}^{+\infty}\int_{\mathbb{R}}|u(x,0,t) - u(x-z,0,t-\tau)|^2\frac{e^{-\frac{|z|^2}{4\tau}}}{\tau^2}dzd\tau\right]u(x,0,t)\\ \quad\quad\quad\quad\quad\quad\quad\quad\quad\quad\quad\quad\quad\quad\quad\quad\quad\quad\quad\quad\quad\text{ for all }(x,0,t)\in \partial\mathbb{R}^2_+\times (0, T),\\
u(x, y, t) = u_0(x, y)\text{ for all }(x, y, t)\in \mathbb{R}^2_+\times (-\infty, 0].
  \end{cases}
\end{equation*}

\section{The outer-inner gluing scheme}
By possibly modifying $b(x, 0, t)$, system (\ref{e:equationforvarphi})-(\ref{e:equationforvarphi1}) can be rewritten as
\begin{equation}\label{e:finalproblem1}
\begin{aligned}
0=\mathcal{E}^*_1 -\partial_t\varphi + \Delta\varphi\text{ in }\mathbb{R}^2_+\times (0, T),
\end{aligned}
\end{equation}
and
\begin{equation}\label{e:finalproblem2}
\begin{aligned}
0 & = \mathcal{E}^*_2 + \frac{d}{dy}\varphi(x, 0, t)\\
&\quad + \frac{2}{1+\left|(u, v)\right|^2}\Pi_{U^\perp}\varphi\\
&\quad + \frac{1}{\pi}\left[\int_{\mathbb{R}}\frac{(U(x,0,t)-U(x-z,0,t))\cdot\left(\Pi_{U^\perp}\varphi(x,0,t) -\Pi_{U^\perp}\varphi(x-z,0,t)\right)}{|z|^2}dz\right]\\
&\quad\quad\quad\quad\quad\quad\quad\quad\quad\quad\quad\quad\quad\quad\quad\quad\quad\quad\quad\quad\quad\quad\quad\quad\quad\quad\quad\quad\quad U(x, 0, t)\\
&\quad + \left[\frac{1}{8\pi}\int_{0}^{+\infty}\int_{\mathbb{R}}|U(x,0,t) - U(x-z,0,t-\tau)|^2\frac{e^{-\frac{|z|^2}{4\tau}}}{\tau^2}dzd\tau - \frac{2}{1+\left|(u, v)\right|^2}\right]\\
&\quad\quad\quad\quad\quad\quad\quad\quad\quad\quad\quad\quad\quad\quad\quad\quad\quad\quad\quad\quad\quad\quad\quad\quad\quad\quad\quad\quad\quad\quad\quad
\Pi_{U^\perp}\varphi\\
&\quad + N_U(\Pi_{U^\perp}(\Phi^*+\varphi))+b(x, 0, t)U\text{ in }\partial\mathbb{R}^2_+\times (0, T).
\end{aligned}
\end{equation}
Here and in the rest of this paper, we use the notation $u = \frac{x-\xi(t)}{\lambda(t)}$ and $v = \frac{y}{\lambda(t)}$.
Let $\eta_0(s)$ be a smooth cut-off function with $\eta_0(s) = 1$ for $s < 1$ and $=0$ for $s > \frac{3}{2}$. Consider an increasing function $R(t)$ satisfying
$$
R(t) > 0, \quad R(t)\to  \infty\text{ as } t\nearrow T
$$
and define
$$
\eta(x, t):=\eta_0\left(\frac{x-\xi(t)}{R(t)\lambda(t)}\right), \quad \tau_\lambda(t) = \tau_0 + \int_{0}^t\frac{ds}{\lambda^2(s)}
$$
such that
\begin{equation*}
\tau_\lambda\sim \tau_0 + \frac{1}{\lambda_0}\frac{\log^2(T-t)}{\log T}.
\end{equation*}
We decompose the function $\varphi(x, y, t)$ into the following form
\begin{equation}\label{e:errorterm}
\varphi(x, y, t) = \eta \phi\left(\frac{x-\xi(t)}{\lambda(t)}, \frac{y}{\lambda(t)}, \tau_\lambda(t)\right) + \psi(x, y, t)
\end{equation}
with $\phi(u, v, \tau) = 0$ for $\tau\in (-\infty, \tau_0]$ and $\phi(u, 0, \tau)\cdot \omega(u, 0) \equiv 0$ for all $\tau\in (\tau_0, +\infty)$.
Then $\varphi(x, y, t)$ given by (\ref{e:errorterm}) solves (\ref{e:finalproblem1})-(\ref{e:finalproblem2}) if the pair $(\phi, \psi)$ satisfies the following system of evolution equations
\begin{equation}\label{e:innerproblem}
\left\{
\begin{array}{ll}
\partial_\tau\phi = \Delta\phi + \chi_{\mathcal{D}_{2R}}\lambda^2\mathcal{E}^*_1\text{ in }\mathbb{R}^2_+\times (\tau_0, +\infty),\\\\
-\frac{d}{dv}\phi(u, 0, \tau) = \frac{2}{1+\left|u\right|^2}\phi\\
\quad + \frac{1}{\pi}\left[\int_{\mathbb{R}}\frac{(\omega(u,0)-\omega(u-z,0))\cdot\left(\phi(u,0,\tau) -\phi(u-z,0,\tau)\right)}{|z|^2}dz\right]\omega(u, 0)\\
\quad + \chi_{\mathcal{D}_{2R}\cap\left(\partial\mathbb{R}^2_+\times (\tau_0, +\infty)\right)}\left(\lambda\Pi_{\omega^\perp}\mathcal{E}^*_2 + \frac{2}{1+\left|u\right|^2}\Pi_{\omega^\perp}\psi\right)\\
\quad + \chi_{\mathcal{D}_{2R}\cap\left(\partial\mathbb{R}^2_+\times (\tau_0, +\infty)\right)}\\
\quad\quad \times\left(\frac{1}{\pi}\int_{\mathbb{R}}\frac{(\omega(u,0)-\omega(u-z,0))\cdot\left(\Pi_{\omega^\perp}\psi(u,0,\tau) -\Pi_{\omega^\perp}\psi(u-z,0,\tau)\right)}{|z|^2}dz\right)\omega(u, 0)\\
\quad\quad\quad\quad\quad\quad\quad\quad\quad\quad\quad\quad\quad\quad\quad\quad\quad\quad\quad\quad\quad \text{ in }\partial\mathbb{R}^2_+\times (\tau_0, +\infty),\\
\phi = 0\text{ in }\mathbb{R}^2_+\times(-\infty, \tau_0]
\end{array}
\right.
\end{equation}
and
\begin{equation}\label{e:outerproblem}
\left\{
\begin{array}{ll}
\partial_t\psi = \Delta\psi - \frac{\dot{\lambda}}{\lambda}(u, v)\cdot \nabla_{(u, v)}\phi - \frac{\dot{\xi}}{\lambda}\frac{d\phi}{du}\\
\quad\quad\quad + [\Delta\eta\phi + \frac{2}{\lambda}\nabla\eta\nabla\phi - \partial_t\eta\phi] + (1-\eta)\mathcal{E}^*_1\text{ in }\mathbb{R}^2_+\times (0, T), \\
-\frac{d}{dy} \psi(x, 0, t)\\
\quad\quad\quad = (1-\eta)\frac{2\lambda}{\lambda^2+|x-\xi(t)|^2}\psi + (1-\eta)\Pi_{U^\perp}\mathcal{E}^*_2 \\
\quad\quad\quad\quad + \left(\frac{d}{dy}\eta\right)(x, 0, t)\phi + N_U(\Pi_{U^\perp}(\Phi^*+\varphi))\text{ in }\partial\mathbb{R}^2_+\times (0, T),\\
\psi = \psi_0 \text{ in }\mathbb{R}^2_+\times(-\infty, 0].
\end{array}
\right.
\end{equation}
Here $\psi_0$ is a small function which will be determined later, $\chi_A$ is the characteristic function of the set $A$, i.e., $\chi(z) = 1$ if $z\in A$, $\chi(z) = 0$ if $z\not\in A$, $\psi_\infty$ is defined by
$$
\psi_{\infty}[\lambda, \xi] = (u_{\infty} - U) - \Pi_{U^\perp}[\varphi^*] + \tilde{a}(u_\infty-U, \Pi_{U^\perp}[\varphi^*]),
$$
$$
u_{\infty}(x, y) = \frac{Z^*_0(x, y)+\vec{e}}{|Z^*_0(x, y)+\vec{e}|},\quad \vec{e} = \begin{pmatrix}0\\ 1\end{pmatrix}.
$$
Here $\tilde{a}(\cdot, \cdot)$ is determined by the following nonlinear equation
$$
\Pi_{U^\perp}[\varphi] = (u_{\infty} - U) - \Pi_{U^\perp}[\varphi^*] + a(\Pi_{U^\perp}[\varphi^*] + \Pi_{U^\perp}[\varphi])U.
$$
Here we also define the set $\mathcal{D}_{\gamma R} = \{(u, v, \tau)|\tau\in (\tau_0, +\infty), (u, v)\in \mathbb{R}^2_+, |(u, v)|\leq \gamma R\}$ for $\gamma > 0$.

(\ref{e:innerproblem}) is the so-called inner problem and (\ref{e:outerproblem}) is the outer problem. This is a highly nonlinear system, we will apply Schauder's fixed point theorem to solve it. To this aim, we need a linear theory of the following equation
\begin{equation*}
\left\{
\begin{array}{ll}
\partial_\tau\phi = \Delta\phi + \chi_{\mathcal{D}_{2R}}\lambda^2\mathcal{E}^*_1\text{ in }\mathbb{R}^2_+\times (\tau_0, +\infty),\\\\
-\frac{d}{dv}\phi(u, 0, \tau) = \frac{2}{1+\left|u\right|^2}\phi\\
\quad + \frac{1}{\pi}\left[\int_{\mathbb{R}}\frac{(\omega(u,0)-\omega(u-z,0))\cdot\left(\phi(u,0,\tau) -\phi(u-z,0,\tau)\right)}{|z|^2}dz\right]\omega(u, 0)\\
\quad + G[\lambda, \xi, \psi](u, 0, \tau)\text{ in }\partial\mathbb{R}^2_+\times (\tau_0, +\infty),\\
\phi = 0\text{ in }\mathbb{R}^2_+\times(-\infty, \tau_0]
\end{array}
\right.
\end{equation*}
where
\begin{equation*}
\begin{aligned}
&G[\lambda, \xi, \psi]\\
&= \chi_{\mathcal{D}_{2R}\cap\left(\partial\mathbb{R}^2_+\times (\tau_0, +\infty)\right)}\left(\lambda\Pi_{\omega^\perp}\mathcal{E}^*_2 + \frac{2}{1+\left|u\right|^2}\Pi_{\omega^\perp}\psi\right)\\
&\quad +\chi_{\mathcal{D}_{2R}\cap\left(\partial\mathbb{R}^2_+\times (\tau_0, +\infty)\right)}\times\\
&\left(\frac{1}{\pi}\int_{\mathbb{R}}\frac{(\omega(u,0)-\omega(u-z,0))\cdot\left(\Pi_{\omega^\perp}\psi(u,0,\tau) -\Pi_{\omega^\perp}\psi(u-z,0,\tau)\right)}{|z|^2}dz\right)\omega(u, 0).
\end{aligned}
\end{equation*}
In Section 4, we will construct a solution $\phi$ of the following equation
\begin{equation}\label{e:innerproblemlinear}
\left\{
\begin{array}{ll}
\partial_\tau\phi = \Delta\phi + g(u, v, \tau)\text{ in }\mathbb{R}^2_+\times (\tau_0, +\infty),\\\\
-\frac{d}{dv}\phi(u, 0, \tau) = \frac{2}{1+\left|u\right|^2}\phi+ \mathcal{A}[\phi] + h(u, \tau)\text{ in }\partial\mathbb{R}^2_+\times (\tau_0, +\infty),\\
\phi = 0\text{ in }\mathbb{R}^2_+\times(-\infty, \tau_0],
\end{array}
\right.
\end{equation}
which defines a bounded linear operator of the functions $g$ (with compact support in $\mathcal{D}_{2R}$) and $h$ (with compact support in $\mathcal{D}_{2R}\cap \left(\partial\mathbb{R}^2_+\times (\tau_0, +\infty)\right)$) satisfying good $L^\infty$-weight estimates when certain further orthogonality conditions hold.
Here and in the following, we use the notation
$$
\mathcal{A}[\phi]: = \frac{1}{\pi}\left[\int_{\mathbb{R}}\frac{(\omega(u,0)-\omega(u-z,0))\cdot\left(\phi(u,0,\tau) -\phi(u-z,0,\tau)\right)}{|z|^2}dz\right]\omega(u, 0).
$$
In Section 5, we use Schauder's fixed point theorem to prove the existence of solution for (\ref{e:innerproblem}) and (\ref{e:outerproblem}). This provides a solution to (\ref{e:main}) and Theorem \ref{t:main} is concluded.

\section{Linear theory for the inner problem}
In this section, we consider (\ref{e:innerproblemlinear}). Our aim is to construct a solution for (\ref{e:innerproblemlinear}) which defines a bounded linear operator of $g$, $h$ and satisfies good bounds in suitable weighted norms. We divide the discussion into two cases.
\begin{enumerate}
\item[$\bullet$]Case 1. The first component of the vector-valued function $\phi(u, v, \tau)$ is odd in the variable $u$, the second component of the vector-valued function $\phi(u, v, \tau)$ is even in the variable $u$. Correspondingly, we assume the first components of the vector-valued functions $g(u, v, \tau)$ and $h(u, \tau)$ are odd in the variable $u$, the second components of the vector-valued functions $g(u, v, \tau)$ and $h(u, \tau)$ are even in the variable $u$.
\item[$\bullet$]Case 2. The first component of the vector-valued function $\phi(u, v, \tau)$ is even in the variable $u$, the second component of the vector-valued function $\phi(u, v, \tau)$ is odd in the variable $u$. Correspondingly, we assume the first components of the vector-valued functions $g(u, v, \tau)$ and $h(u, \tau)$ are even in the variable $u$, the second components of the vector-valued functions $g(u, v, \tau)$ and $h(u, \tau)$ are odd in the variable $u$.
\end{enumerate}
\subsection{Case 1.}
This subsection is devoted to construct a solution to the initial value problem
\begin{equation}\label{p11}
\left\{
\begin{array}{lll}
\partial_\tau\phi = \Delta\phi + g(u, v, \tau)\text{ in }B^+_{2R(\tau)}\times (\tau_0, +\infty),\\
-\frac{d}{dv}\phi = L_\omega[\phi] + h(u, \tau)\text{ in }\left(B^+_{2R(\tau)}\times (\tau_0, +\infty)\right)\cap
\left(\partial\mathbb{R}^2_+\times (\tau_0, +\infty)\right),\\
\phi(u, v, \tau) = 0\text{ in }B^+_{2R(\tau)}\times (-\infty, \tau_0],\\
\phi \cdot \omega = 0 \text{ in } \left(B^+_{2R(\tau)}\times (\tau_0, +\infty)\right)\cap \left(\partial\mathbb{R}^2_+\times (\tau_0, +\infty)\right).
\end{array}
\right.
\end{equation}
for any given functions $g$, $h$ with $\|g\|_{a, \nu} <+\infty$, $\|h\|_{a, \nu} <+\infty$, the first components of $g$ and $h$ are even in the $u$ variable, we use the idea from \cite{cortazar2016green} and \cite{davila2017singularity}.
\begin{prop} \label{prop1}
Let $1< a < 2$ and $\nu > 0$ be given positive numbers. Then, for any $g$, $h$ with $\|g\|_{a, \nu} <+\infty$, $\|h\|_{a, \nu} <+\infty$, the first components of $g$ and $h$ are odd in the $u$ variable, the second components of $g$ and $h$ are even in the $u$ variable, and satisfying \begin{equation}\label{e:orthognalconditionz3}
\int_{B^+_{2R}}g(u, v, \tau)\cdot Z_3(u, v)dudv + \int_{-2R}^{2R}h(u,\tau)\cdot Z_3(u)du = 0\quad\text{for all }\tau\in (\tau_0,\infty)
\end{equation}
there exist  $\phi = \phi[g, h]$ solving (\ref{p11}) which defines a bounded linear operator of $g$, $h$. Furthermore, the following estimate holds
\begin{equation*}
\begin{aligned}
&|\phi[g, h]|\lesssim\tau^{-\nu}R^2\log^2 R\left(\frac{\|h\|_{a,\nu}}{1+|u|^a} + \frac{R^{2-a}\|g\|_{a,\nu}}{1+|(u, v)|^2}\right).
\end{aligned}
\end{equation*}
\end{prop}

\begin{proof}[{\bf Proof of Proposition \ref{prop1}}] We divide the proof into two steps. First, we construct a solution to (\ref{p11}) with zero boundary condition on
$\mathbb{R}^2_+\setminus B^+_{2R(\tau)}$ and for $g$, $h$ not necessarily satisfying condition (\ref{e:orthognalconditionz3}). Then, we use of this construction to solve \eqref{p11}.

{\it {\bf Step $1$}}.
We claim that for any $G$, $H$ satisfying $\|G\|_{b,\nu} <+\infty$, $\|H\|_{c,\nu} <+\infty$, $b\in (-1, 2)$, $c\in (-1, 1)$, there exists $\phi = \phi (u, v, \tau )$ solving
\begin{equation}\label{modo0}
\left\{
\begin{array}{lll}
\partial_\tau\phi = \Delta\phi + G(u, v, \tau)\text{ in }B^+_{2R(\tau)}\times (\tau_0, +\infty),\\
-\frac{d}{dv}\phi = L_\omega[\phi] + H(u,\tau)\text{ in }\left(B^+_{2R(\tau)}\times (\tau_0, +\infty)\right)\cap
\left(\partial\mathbb{R}^2_+\times (\tau_0, +\infty)\right),\\
\phi = 0\text{ on }\left(\mathbb{R}^2_+\setminus B_{2R(\tau)}^+(0)\right)\times (\tau_0, +\infty),\\
\phi(u, v, \tau) = 0\text{ in }B^+_{2R(\tau)}\times (-\infty, \tau_0]
\end{array}
\right.
\end{equation}
and satisfying
\begin{equation*}
\begin{aligned}
(1+|(u, v)|)&|\nabla\phi(u, v, \tau)|+|\phi(u, v,\tau)|\\
&\lesssim\tau^{-\nu}R^2\log^2 R(R^{2-b}\|G\|_{b, \nu}+R^{1-c}\|H\|_{c, \nu}).
\end{aligned}
\end{equation*}

Let $\eta(s)$ be a smooth cut-off function, for a fixed but large number $\ell$ independent from $R$, we define $\eta_\ell(u, v) = \eta (|(u, v)|-\ell )$.
From standard parabolic theory, there exists a unique solution $\phi_*[G, H]$ of
\begin{equation*}
\left\{
\begin{array}{lll}
\partial_\tau\phi = \Delta\phi + G(u, v, \tau)\text{ in }B^+_{2R(\tau)}\times (\tau_0, +\infty),\\
-\frac{d}{dv}\phi = L_\omega[(1-\eta_{\ell})\phi] + H(u,\tau)\text{ in }\left(B^+_{2R(\tau)}\times (\tau_0, +\infty)\right)\cap
\left(\partial\mathbb{R}^2_+\times (\tau_0, +\infty)\right),\\
\phi = 0\text{ on }\left(\mathbb{R}^2_+\setminus B_{2R(\tau)}^+(0)\right)\times (\tau_0, +\infty),\\
\phi(u, v, \tau) = 0\text{ in }B^+_{2R(\tau)}\times (-\infty, \tau_0].
\end{array}
\right.
\end{equation*}
The first component of $\phi_*[G, H]$ is even in the $u$ variable and satisfies
$$
\left|\phi_*[G, H]\right|\lesssim\tau^{-\nu}(R^{2-b}\|G\|_{b, \nu}+R^{1-c}\|H\|_{c, \nu}).
$$
Setting $\phi=\phi_*[G, H]+\tilde\phi $, then (\ref{modo0}) is reduced to the following problem
\begin{equation}\label{modo01}
\left\{
\begin{array}{lll}
\partial_\tau\tilde{\phi} = \Delta\tilde{\phi} + G(u, v, \tau)\text{ in }B^+_{2R(\tau)}\times (\tau_0, +\infty),\\
-\frac{d}{dv}\tilde{\phi} = L_\omega[\tilde{\phi}] + \tilde H_0(u,\tau)\text{ in }\left(B^+_{2R(\tau)}\times (\tau_0, +\infty)\right)\cap
\left(\partial\mathbb{R}^2_+\times (\tau_0, +\infty)\right),\\
\tilde{\phi} = 0\text{ on }\left(\mathbb{R}^2_+\setminus B_{2R(\tau)}^+(0)\right)\times (\tau_0, +\infty),\\
\tilde{\phi}(u, v, \tau) = 0\text{ in }B^+_{2R(\tau)}\times (-\infty, \tau_0],
\end{array}
\right.
\end{equation}
where $\tilde H_0 = \frac{2}{1+|u|^2}\eta_\ell\phi_*[G, H]$. Notice that the first component of $\tilde H_0$ is even in $u$ variable and it is compactly supported with size controlled by $G$ and $H$. Hence, for any $m>0$, we have
\begin{equation}\label{tete}
\begin{aligned}
|\tilde H_0(u, v, \tau)|&\lesssim\frac{\tau^{-\nu}}{1+|(u,v)|^m}\left[\sup_{\tau > \tau_0}\tau^{\nu}|\phi_*[G, H](\cdot,\tau)|\right]\\
&\lesssim\frac{\tau^{-\nu}}{1+|(u, v)|^m}(R^{2-b}\|G\|_{b,\nu}+R^{1-c}\|H\|_{c, \nu}).
\end{aligned}
\end{equation}
Testing (\ref{modo01}) against $\tilde \phi$ and integrating, we obtain
$$
\partial_\tau\int_{B_{2R}^+}\tilde\phi^2+Q(\tilde \phi,\tilde \phi)=\int_{B_{2R}}G\tilde\phi+\int_{-2R}^{2R}\tilde H_0\tilde \phi,
$$
here $Q$ is the quadratic form defined by
\begin{equation*}
Q(\phi,\phi):=\int_{B_{2R}^+}|\nabla \phi|^2dudv - \int_{-2R}^{2R}\frac{2}{1+|u|^2}|\phi|^2du.
\end{equation*}
It is easy to check that there exists a constant $ \beta >0$ such that, for any $\phi$ with $\int_{B_{2R}^+}\phi\cdot Z_3dudv = 0$ and $\phi = 0$ on $\mathbb{R}^2_+\setminus B_{2R}^+$, we have
$$
Q(\phi,\phi)\geq\frac{\beta}{R^2\log R} \int_{B_{2R}^+}\phi^2dudv.
$$
Thus for some $\beta' >0$, thereholds
\begin{equation}\label{en}
\partial_\tau\int_{B_{2R}^+}\tilde\phi^2+\frac{\beta'}{R^2\log R}\int_{B_{2R}^+}\tilde\phi^2\lesssim R^2\log R
\left(\int_{B_{2R}^+} G^2 + \int_{-2R}^{2R} \tilde{H}^2_0\right).
\end{equation}
Set
$$
K:=\left[\sup_{\tau>\tau_0}\tau^{\nu}\|\phi_*[G, H](\cdot,\tau)\|_{L^\infty}\right].
$$
On the other hand, using estimate (\ref{tete}) for a large $m$, we obtain
$$
\left(\int_{B_{2R}^+}G^2 + \int_{-2R}^{2R}\tilde{H}^2_0\right)\lesssim\tau^{-2\nu}K^2.
$$
By the fact that $\tilde\phi(\cdot, \tau_0) =0$ and Gronwall's inequality, we obtain from (\ref{en}) that
\begin{equation*}
\|\tilde\phi(\cdot,\tau)\|_{L^2(B_{2R}^+)}\lesssim\tau^{-\nu}KR^2\log R,
\end{equation*}
for all $\tau > \tau_0$. From standard parabolic estimates, we get
$$
\|\tilde\phi(\cdot,\tau)\|_{L^\infty(B_M)}\lesssim\tau^{-\nu }KR^2\log R\text{ for all }\tau>\tau_0.
$$
Therefore,
\begin{equation*}
\begin{aligned}
(1+|(u, v)|)|\nabla\tilde\phi(u, v, \tau)|&+|\tilde\phi(u, v, \tau)|\\
&\lesssim\tau^{-\nu}R^2\log^2 R\left[\sup_{\tau > \tau_0}\tau^{\nu}|\phi_*[G, H] (\cdot, \tau)|\right].
\end{aligned}
\end{equation*}
From this estimate and (\ref{tete}), the function $\phi_0[G, H] :=\tilde\phi + \phi_*[G, H]$ solves (\ref{modo0}) and satisfies
\begin{equation*}
\begin{aligned}
&(1+|(u, v)|)|\nabla\phi_0(u, v, \tau)|+|\phi_0(u, v, \tau)|\\
&\quad\quad\quad\quad\quad\quad\lesssim\tau^{-\nu}R^2\log^2 R(R^{2-b}\|G\|_{b, \nu}+R^{1-c}\|H\|_{c, \nu}).
\end{aligned}
\end{equation*}

{\it {\bf Step $2$}.}
For bounded functions $g = g(u, v)$, $h=h(u)$ in $B^+_{2R}$ whose first components are even in the $u$ variable and $\int_{B^+_{2R}}g(u, v, \tau)\cdot Z_3(u, v)dudv + \int_{-2R}^{2R}h(u,\tau)\cdot Z_3(u)du = 0$. Let us extent $h$ as zero outside $B_{2R}^+$ and still denote the extended function as $h$.
From standard elliptic estimate, the equation
\begin{equation*}
\left\{
\begin{array}{lll}
\Delta\phi =  g(u, v, \tau)\text{ in }\mathbb{R}^2_+,\\
-\frac{d}{dv}\phi = L_\omega[\phi] + h(u, \tau)\text{ in }\partial\mathbb{R}^2_+,\\
\lim_{|(u, v)|\to +\infty}\phi(u, v) = 0.
\end{array}
\right.
\end{equation*}
has a solution $H =: L_0^{-1}[g, h] $ satisfying
$$
|H(u, v,\tau)|\lesssim\tau^{-\nu}\left(\frac{1}{(1+|(u, v)|)^{a-1}}\|h\|_{a,\nu}+\frac{1}{(1+|(u, v)|)^{a-2}}\|g\|_{a,\nu}\right).
$$
Let $\Phi_0$ be the unique solution in $B_{3R}^+$ of the problem
\begin{equation*}
\left\{
\begin{array}{lll}
\partial_\tau\phi = \Delta\phi + H(u, v, \tau)\text{ in }B^+_{3R(\tau)}\times (\tau_0, +\infty),\\
-\frac{d}{dv}\phi = L_\omega[\phi] + H(u,0,\tau)\text{ in }\left(B^+_{3R(\tau)}\times (\tau_0, +\infty)\right)\cap
\left(\partial\mathbb{R}^2_+\times (\tau_0, +\infty)\right),\\
\phi = 0\text{ on }\left(\mathbb{R}^2_+\setminus B_{3R(\tau)}^+(0)\right)\times (\tau_0, +\infty),\\
\phi(u, v, \tau) = 0\text{ in }B^+_{3R(\tau)}\times (-\infty, \tau_0].
\end{array}
\right.
\end{equation*}
From Step 1, $\Phi_0[H]$ defines a bounded linear operator of $H$ and satisfies the estimate
\begin{equation*}
\begin{aligned}
|\Phi_0(u,v,\tau)|\lesssim\tau^{-\nu}R^{2}\log^2 R\left(R^{4-a}\|H(u,v,\tau)\|_{a-2,\nu}+R^{2-a}\|H(u,0,\tau)\|_{a-1,\nu}\right).
\end{aligned}
\end{equation*}

Now let us fix a vector $e$ with $|e|=1$, a large number $\rho>0$ with $\rho \leq 2R$ and $\tau_1 \geq \tau_0$. Consider the following change of variables
\begin{equation*}
\Phi_\rho (z,t):=\Phi_0(\rho e+\rho z, \tau_1+\rho^2 t),
\end{equation*}
\begin{equation*}
G_\rho (z,t) := \rho^2H(\rho e+\rho z, \tau_1+\rho^2 t),
\end{equation*}
\begin{equation*}
H_\rho (z,t) := \rho H(\rho e+\rho z, 0, \tau_1+\rho^2 t).
\end{equation*}
Then $\Phi_\rho(z, t)$ satisfies
\begin{equation*}
\left\{
\begin{array}{lll}
\partial_\tau\Phi_\rho = \Delta_z\Phi_\rho + B_\rho(z,t)\Phi_\rho+G_\rho(z,t),\quad (z, t)\in B_{1}^+(0)\times (0, 2),\\
-\frac{d}{dz_2}\phi = C_\rho(z,t)\Phi_\rho+H_\rho(z,t),\quad (z, t)\in (-1,1)\times \{0\}\times (0, 2)
\end{array}
\right.
\end{equation*}
with $B_\rho = O(\rho^{-2})$, $C_\rho = O(\rho^{-2})$ uniformly in $B_1^+(0) \times (0,\infty)$. From standard parabolic estimates, we have
\begin{equation*}
\begin{aligned}
\|\nabla_z\Phi_\rho\|_{L^\infty(B^+_{\frac 12 }(0)\times(1,2))} &\lesssim\|\Phi_\rho\|_{L^\infty(B_{1}^+(0)\times(0, 2))}\\
&\quad + \|G_\rho\|_{L^\infty(B_{1}^+(0)\times(0, 2))}+\|H_\rho\|_{L^\infty((-1,1)\times \{0\}\times (0, 2))}.
\end{aligned}
\end{equation*}
Furthermore, there holds
$$
\|G_\rho\|_{L^\infty(B_{1}^+(0)\times(0, 2))}\lesssim\rho^{1-a}\tau_1^{-\nu}\|H\|_{a,\nu},
$$
$$
\|H_\rho\|_{L^\infty((-1,1)\times\{0\}\times (0, 2))}\lesssim\rho^{1-a}\tau_1^{-\nu}\|H\|_{a,\nu},\quad
\|\Phi_\rho\|_{L^\infty(B_{1}^+(0)\times(0, 2))}\lesssim\tau_1^{-\nu} K(\rho)
$$
with
\begin{equation}\label{Knew}
K(\rho) = R^{2}\log^2 R\left(R^{2-a}\|h\|_{a,\nu}+R^{4-a}\|g\|_{a,\nu}\right).
\end{equation}
Hence
$$
\rho |\nabla\Phi_0(\rho e, \tau_1+\rho^2)|\lesssim\tau_1^{-\nu}K(\rho) .
$$
Choose $\tau_0\ge R^2$, then we have
\begin{equation*}\label{coco}
(1+|(u, v)|)|\nabla\Phi_0(u, v, \tau)|\lesssim\tau^{-\nu}K(|(u, v)|)
\end{equation*}
for any $\tau > 2\tau_0$ and $|(u, v)|\le 3R$.

Since $H$ is of class $C^1$ and $\|\nabla H\|_{a-1,\nu}\leq\|h\|_{a,\nu}+\|g\|_{a,\nu}$, we obtain
$$
(1+|(u, v)|^2)|D^2\Phi_0(u, v, \tau)|\lesssim\tau^{-\nu} K(|(u, v)|)
$$
for all $\tau> \tau_0$, $|(u, v)|\leq 2R$ with $K$ being defined in (\ref{Knew}). Thus we have
\begin{equation*}
\begin{aligned}
(1+|(u, v)|^2)|&D^2\Phi_0(u, v, \tau)|+(1+|(u, v)|)|\nabla\Phi_0(u, v, \tau)|+|\Phi_0(u, v,\tau)|\\
&\lesssim\tau^{-\nu}R^{2}\log^2 R\left(R^{2-a}\|h\|_{a,\nu}+R^{4-a}\|g\|_{a,\nu}\right).
\end{aligned}
\end{equation*}
Therefore
\begin{equation*}
\begin{aligned}
&|L_0[\Phi_0](\cdot,\tau)|\lesssim\tau^{-\nu}R^{2}\log^2 R\left(\frac{R^{2-a}\|h\|_{a,\nu}}{1+|(u, v)|}+\frac{R^{4-a}\|g\|_{a,\nu}}{1+|(u, v)|^{2}}\right).
\end{aligned}
\end{equation*}
Define
\begin{equation*}
\phi_0[g, h] := L_0 [\Phi_0].
\end{equation*}
Then $\phi_0[g, h]$ satisfies (\ref{p11}) and the proof is completed.
\end{proof}
\subsection{Case 2.}
The following proposition is valid.
\begin{prop}\label{proposition4.1}
Let $1 < a < 2$, $\nu > 0$ be given positive numbers. Then, for $R > 0$ sufficiently large and any $g = g(u, v, \tau)$, $h = h(u, \tau)$ with $\|g\|_{a, \nu} < +\infty$, $\|h\|_{a, \nu} < +\infty$, the first components of $g(u, v, \tau)$ and $h(u, \tau)$ are even in the $u$ variable for all $\tau$, the second components of $g(u, v, \tau)$ and $h(u, \tau)$ are odd in the $u$ variable for all $\tau$, and satisfying
\begin{equation*}
\int_{B^+_{2R(\tau)}}g(u, v, \tau)\cdot Z_2(u, v)dudv + \int_{-2R}^{2R}h(u,\tau)\cdot Z_2(u, 0)du = 0\quad\text{for all }\tau\in (\tau_0,\infty),
\end{equation*}
there exists $\phi = \phi[g, h]$ solving (\ref{e:innerproblemlinear}), which defines a linear operator of $g$ and $h$ satisfying
\begin{equation*}
|\phi(u, v, \tau)|\lesssim \tau^{-\nu}(1+|(u, v)|)^{-\sigma}\left(R^{2+\sigma-a}\|g\|_{a,\nu} + R^{1+\sigma-a}\|h\|_{a,\nu}\right)
\end{equation*}
for some $\sigma\in (0, 1)$.
\end{prop}
To prove this proposition, first we consider the following problem in the whole half space
\begin{equation}\label{e5:32}
\left\{
\begin{array}{lll}
\partial_\tau\phi = \Delta\phi + g(u, v, \tau)\text{ in }\mathbb{R}^2_+\times (\tau_0, +\infty),\\
-\frac{d}{dv}\phi = \frac{2}{1+\left|u\right|^2}\phi + \mathcal{A}[\tilde{\phi}] + h(u, \tau)\text{ in }\left(\mathbb{R}^2_+\times (\tau_0, +\infty)\right)\cap
\left(\partial\mathbb{R}^2_+\times (\tau_0, +\infty)\right),\\
\phi(u, v, \tau) = 0\quad\text{in }\mathbb{R}^2_+\times (-\infty, \tau_0],\\
\phi \cdot \omega = 0 \text{ in } \left(\mathbb{R}^2_+\times (\tau_0, +\infty)\right)\cap \left(\partial\mathbb{R}^2_+\times (\tau_0, +\infty)\right).
\end{array}
\right.
\end{equation}
Then we have
\begin{lemma}\label{l5:3}
Let $0 < \sigma < 1$, $\nu > 0$ be given positive numbers. Then, for $R > 0$ sufficiently large and any $g = g(u, v, \tau)$, $h = h(u, \tau)$ with $\|g\|_{2+\sigma, \nu} < +\infty$, $\|h\|_{1+\sigma, \nu} < +\infty$, the first components of $g(u, v, \tau)$ and $h(u, \tau)$ are even in the $u$ variable for all $\tau$, the second components of $g(u, v, \tau)$ and $h(u, \tau)$ are odd in the $u$ variable for all $\tau$, and satisfying
\begin{equation*}
\int_{\mathbb{R}^2_+}g(u, v, \tau)\cdot Z_2(u, v)dudv + \int_{\mathbb{R}}h(u,\tau)\cdot Z_2(u, 0)du = 0\text{ for all }\tau\in (\tau_0,\infty).
\end{equation*}
Then for sufficiently large $\tau_1 > \tau_0$, the solution of (\ref{e5:32}) satisfies
\begin{equation}\label{e5:35}
\|\phi(u, v, \tau)\|_{\sigma,\tau_1}\lesssim \|g\|_{2+\sigma, \tau_1} + \|h\|_{1+\sigma, \tau_1}.
\end{equation}
Here, $\|g\|_{b, \tau_1}:=\sup_{\tau\in (\tau_0,\tau_1)}\tau^\nu\|(1+|(u, v)|^b)g\|_{L^\infty(\mathbb{R}^2_+)}$.
\end{lemma}
\begin{proof}
First, we claim that $\|\phi\|_{\sigma, \tau_1} < +\infty$ holds for any given $\tau_1 > \tau_0$. Given $R > 0$ there exists a $K = K(R,\tau_1) > 0$ such that
\begin{equation*}
|\phi(u, v,\tau)|\leq K \quad\text{in }B_R(0)\times (\tau_0, \tau_1].
\end{equation*}
Fix $R > 0 $ and $K_1 > 0$ sufficiently large, $K_1\rho^{-\sigma}$ ($\rho = |(u, v)|$) is a super-solution for (\ref{e5:32}). Therefore $|\phi|\leq 2K_1\rho^{-\sigma}$ and $\|\phi\|_{\sigma,\tau_1} < +\infty$ for any $\tau_1 > 0$. We claim that
\begin{equation}\label{e5:33}
\int_{\mathbb{R}^2_+}\phi(u, v, \tau)\cdot Z_2(u, v)dudv = 0\text{ for all }\tau\in (\tau_0,\tau_1).
\end{equation}
Indeed, test the equation against
\begin{equation*}
Z_2\eta,\quad \eta(u, v) = \eta_0(\frac{\rho}{R})
\end{equation*}
with $\eta_0$ being a smooth cut-off function satisfying $\eta_0(r) = 1$ for $r < 1$ and $r = 0$ for $r > 2$, $R$ is a large constant. We obtain
\begin{equation*}
\begin{aligned}
\int_{\mathbb{R}^2_+}\phi(\cdot, \tau)\cdot Z_2\eta & = \int_{0}^\tau ds\left(\int_{\mathbb{R}^2_+}\phi\cdot\Delta(\eta Z_2)dudv + \int_{\mathbb{R}^2_+}g\cdot \eta Z_2dudv \right)\\
&\quad + \int_{0}^\tau ds\left(\int_{\mathbb{R}}\phi\cdot\left(\frac{d}{dv}(\eta Z_2) + \frac{1}{1+|u|^2}\eta Z_2\right)du + \int_{\mathbb{R}}h\cdot \eta Z_2du\right).
\end{aligned}
\end{equation*}
On the other hand, we have
\begin{equation*}
\begin{aligned}
& \int_{\mathbb{R}^2_+}\phi\cdot\Delta(\eta Z_2)dudv + \int_{\mathbb{R}^2_+}g\cdot \eta Z_2dudv \\
&\quad + \int_{\mathbb{R}}\phi\cdot\left(\frac{d}{dv}(\eta Z_2) + \frac{1}{1+|u|^2}\eta Z_2\right)du + \int_{\mathbb{R}}h\cdot \eta Z_2du  = O(R^{-\sigma})
\end{aligned}
\end{equation*}
uniformly on $\tau\in (0,\tau_1)$. Letting $R\to +\infty$, we then have (\ref{e5:33}).

Now we claim that for $\tau_1 > \tau_0$ large enough, any solution $\phi$ of (\ref{e5:32}) with $\|\phi\|_{\sigma,\tau_1} < +\infty$ and (\ref{e5:33}) satisfies the estimate
\begin{equation}\label{e5:34}
\|\phi\|_{\sigma,\tau_1}\lesssim \|h\|_{1+\sigma,\tau_1} + \|g\|_{2+\sigma,\tau_1}.
\end{equation}
Therefore (\ref{e5:35}) is valid.

To prove (\ref{e5:34}), by contradiction, we assume that there exist sequences $\tau_1^k\to +\infty$ and $\phi_k$, $g_k$, $h_k$ satisfying
\begin{equation*}\label{e5:36}
\begin{aligned}
&\partial_\tau \phi_k = \Delta\phi_k + g_k\text{ in }\mathbb{R}^2_+\times (\tau_0, +\infty), \\
&-\frac{d}{dv}\phi_k =  \frac{2}{1+\left|u\right|^2}\phi_k + \mathcal{A}[\phi_k] + h_k\text{ in }\left(\mathbb{R}^2_+\times (\tau_0, +\infty)\right)\cap\left(\partial\mathbb{R}^2_+\times (\tau_0, +\infty)\right),\\
&\int_{\mathbb{R}^2_+}\phi_k(u, v, \tau)\cdot Z_2(u, v)dudv = 0\text{ for all }\tau\in (\tau_0,\tau_1^k), \\
&\phi_k(u, v, \tau) = 0\text{ for } (u, v, \tau)\in \mathbb{R}^2_+\times (-\infty, \tau_0].
\end{aligned}
\end{equation*}
and
\begin{equation}\label{e5:38}
\|\phi_k\|_{\sigma,\tau_1^k} = 1,\quad \|g_k\|_{2+\sigma,\tau_1^k}\to 0,\quad \|h_k\|_{1+\sigma,\tau_1^k}\to 0.
\end{equation}
First we claim that
\begin{equation}\label{e5:37}
\sup_{\tau_0 < \tau < \tau_1^k}\tau^\nu|\phi_k(u, v,\tau)|\to 0
\end{equation}
holds uniformly on compact subsets of $\mathbb{R}^2_+$. If not, for some $|(u_k, v_k)|\leq M$ and $\tau_0 < \tau_2^k < \tau_1^k$, there holds
\begin{equation*}
(\tau_2^k)^\nu(1+|(u_k, v_k)|^\sigma)|\phi(u_k, v_k,\tau_2^k)|\geq \frac{1}{2}.
\end{equation*}
Clearly, $\tau_2^k\to +\infty$. Define
\begin{equation*}
\tilde{\phi}_n(u, v,\tau) = (\tau_2^k)^\nu\phi_n(u, v,\tau_2^k + \tau).
\end{equation*}
Then we have
\begin{equation*}
\partial_\tau\tilde{\phi}_k = \Delta\tilde{\phi}_k + \tilde{g}_k\quad\text{in }\mathbb{R}^2_+\times (\tau_0-\tau_2^k,0]
\end{equation*}
\begin{equation*}
\begin{aligned}
&-\frac{d}{dv}\tilde{\phi}_k = \frac{2}{1+|u|^2}\tilde{\phi}_k + \mathcal{A}[\tilde{\phi}_k] + \tilde{h}_k\text{ in }\left(\mathbb{R}^2_+\times (\tau_0-\tau_2^k,0]\right)\cap
\left(\partial\mathbb{R}^2_+\times (\tau_0-\tau_2^k,0]\right)
\end{aligned}
\end{equation*}
where $\tilde{h}_k\to 0$ uniformly on compact subsets of $\mathbb{R}\times (-\infty, 0]$ and
\begin{equation*}
|\tilde{\phi}_k(u, v,\tau)|\leq \frac{1}{1+|(u, v)|^\sigma}\text{ in }\mathbb{R}^2_+\times (\tau_0-\tau_2^k,0].
\end{equation*}
By parabolic estimates and passing to a subsequence, $\tilde{\phi}_k\to\tilde{\phi}$ uniformly on compact subsets of $\mathbb{R}^2_+\times (-\infty, 0]$, $\tilde{\phi}\neq 0$ and
\begin{equation*}
\begin{aligned}
&\partial_\tau \tilde{\phi} = \Delta\tilde{\phi}\text{ in }\mathbb{R}^2_+\times (-\infty, 0],\\
&-\frac{d}{dv}\tilde{\phi} = \frac{2}{1+|u|^2}\tilde{\phi} + \mathcal{A}[\tilde{\phi}]\text{ in }\mathbb{R}\times (-\infty, 0],\\
&\int_{\mathbb{R}^2_+}\tilde{\phi}(u, v, \tau)\cdot Z_2(u, v)dudv = 0\text{ for all }\tau\in (-\infty, 0],\\
&|\tilde{\phi}(u, v,\tau)|\leq \frac{1}{1+|(u, v)|^\sigma}\quad\text{in }\mathbb{R}^2_+\times (-\infty, 0].
\end{aligned}
\end{equation*}
We prove that $\tilde{\phi} = 0$ from which we get a contradiction. From standard parabolic regularity theory, $\tilde{\phi}(u, v, \tau)$ is smooth. Testing the first equation above with $\tilde{\phi}$ we have
\begin{equation*}
\frac{1}{2}\partial_\tau\int_{\mathbb{R}^2_+}|\tilde{\phi}_\tau|^2 + B(\tilde{\phi}_\tau, \tilde{\phi}_\tau) = 0
\end{equation*}
where
\begin{equation*}
B(\tilde{\phi}, \tilde{\phi}) = \int_{\mathbb{R}^2_+}|\nabla\tilde{\phi}|^2dudv - \int_{\mathbb{R}}\frac{2}{1+|u|^2}\tilde{\phi}^2(u, 0)du.
\end{equation*}
Clearly, $B(\tilde{\phi}, \tilde{\phi})\geq 0$ and there holds
\begin{equation*}
\int_{\mathbb{R}^2_+}|\tilde{\phi}_\tau|^2 = -\frac{1}{2}\partial_\tau B(\tilde{\phi}, \tilde{\phi}) = 0.
\end{equation*}
Therefore
\begin{equation*}
\partial_\tau\int_{\mathbb{R}^2_+}|\tilde{\phi}_\tau|^2 \leq 0,\quad \int_{-\infty}^0d\tau\int_{\mathbb{R}^2_+}|\tilde{\phi}|^2 < +\infty
\end{equation*}
and hence $\tilde{\phi}_\tau = 0$. Thus $\tilde{\phi}$ is independent of $\tau$ and
\begin{eqnarray*}
\Delta\tilde{\phi} = 0\text{ in }\mathbb{R}^2_+\times (-\infty, 0],\\
-\frac{d}{dv}\tilde{\phi} = \frac{2}{1+|u|^2}\tilde{\phi} + \mathcal{A}[\tilde{\phi}]\text{ in }\partial\mathbb{R}^2_+\times (-\infty, 0].
\end{eqnarray*}
Since $\tilde{\phi}$ is bounded, the nondegeneracy result in \cite{sire2017nondegeneracy} implies that $\tilde{\phi}= cZ_2$ for some constant $c$. Since $\int_{\mathbb{R}^2_+}\tilde{\phi}\cdot Z_2dudv = 0$, $\tilde{\phi} = 0$, which is a contradiction. Thus (\ref{e5:37}) holds.
From (\ref{e5:38}), for a certain $(u_n, v_n)$ with $|(u_n, v_n)|\to +\infty$ there holds
\begin{equation*}
(\tau_2^k)^\nu|(u_k, v_k)|^\sigma|\phi_k(u_k, v_k, \tau_2^k)|\geq \frac{1}{2}.
\end{equation*}
Define
\begin{equation*}
\tilde{\phi}_n(z, \tau):=(\tau_2^k)^\nu|(u_k, v_k)|^\sigma\phi_k((u_k, v_k)+|(u_k, v_k)|z,|(u_k, v_k)|\tau + \tau_2^k),
\end{equation*}
we have
\begin{eqnarray*}
\partial_\tau \tilde{\phi}_k = \Delta\tilde{\phi}_k + \tilde{g}_k(z,\tau),\\
-\frac{d}{dv}\tilde{\phi}_k = a_k\tilde{\phi}_k + \tilde{h}_k(z,\tau)
\end{eqnarray*}
with
\begin{equation*}
\tilde{h}_k(z,\tau) = (\tau_2^k)^\nu|(u_k, v_k)|^{1+\sigma}h_k((u_k, v_k)+|(u_k, v_k)|z,|(u_k, v_k)|\tau + \tau_2^k).
\end{equation*}
By the assumption on $h_k$, we obtain
\begin{equation*}
|\tilde{h}_k(z,\tau)| \lesssim o(1)|(\hat{u}_k, \hat{v}_k)+z|^{-1-\sigma}((\tau_2^k)^{-1}|(u_k, v_k)|\tau + 1)^{-\nu}
\end{equation*}
with
\begin{equation*}
(\hat{u}_k, \hat{v}_k) = \frac{(u_k, v_k)}{|(u_k, v_k)|}\to -\hat{e}
\end{equation*}
and $|\hat{e}|= 1$. Thus $\tilde{h}_k(z,\tau)\to 0$ on compact subsets of $\mathbb{R}\setminus\{\hat{e}\}\times (-\infty, 0]$ uniformly. The same property holds for $a_n$. Moreover, $|\tilde{\phi}_k|\geq \frac{1}{2}$ and
\begin{equation*}
|\tilde{\phi}_k(z,\tau)| \lesssim |(u_k, v_k)+z|^{-\sigma}((\tau_2^k)^{-1}|(u_k, v_k)|\tau + 1)^{-\nu}.
\end{equation*}
Therefore, $\tilde{\phi}_k\to \tilde{\phi}\neq 0$ uniformly over compact subsets of $\mathbb{R}\setminus\{\hat{e}\}\times (-\infty,0]$ and
\begin{eqnarray}\label{e5:39}
\partial_\tau \tilde{\phi} = \Delta\tilde{\phi}\text{ in }\mathbb{R}^2_+\times (-\infty,0],\\
-\frac{d}{dv}\tilde{\phi} = 0\quad\text{in }\mathbb{R}\setminus\{\hat{e}\}\times (-\infty,0],
\end{eqnarray}
\begin{equation}\label{e5:40}
|\tilde{\phi}(z,\tau)|\leq |z-\hat{e}|^{-\sigma}\text{ in }\mathbb{R}^2_+\setminus\{\hat{e}\}\times (-\infty,0].
\end{equation}
Note that $\tilde{\phi}$ is of form $\tilde{\phi} = \begin{pmatrix}\tilde{\phi}_1\\ \tilde{\phi}_2\end{pmatrix} = \begin{pmatrix}\tilde{\phi}_1\\ 0\end{pmatrix}$ and $\tilde{\phi}_1$ is odd in the $u$ variable. By Lemma \ref{l4.3}, functions $\tilde{\phi}$ satisfying (\ref{e5:39})-(\ref{e5:40}) is zero, which is a contradiction.
This completes the proof.
\end{proof}
\begin{lemma}\label{l4.3}
Let $\phi = \phi(u, v, \tau)$ be a scalar solution of
\begin{equation}\label{e:limiteequation}
\left\{
\begin{array}{lll}
\partial_\tau \phi = \Delta\phi\text{ in }\mathbb{R}^2_+\times (-\infty,0],\\
-\frac{d}{dv}\phi = 0 \text{ in }\partial\mathbb{R}^2_+\setminus\{(0,0)\}\times (-\infty,0],\\
|\phi(u, v,\tau)|\leq |(u, v)|^{-\sigma}\text{ in }\mathbb{R}^2_+\setminus\{(0,0)\}\times (-\infty,0],
\end{array}
\right.
\end{equation}
for $0<\sigma<1$ small enough, $\phi(u, v, \tau)$ is odd in the variable $u$ for all $v$ and $\tau$, then $\phi \equiv 0$ on $\mathbb{R}^2_+\times (-\infty,0]$.
\end{lemma}
\begin{proof}
Inspired by the proof of Lemma 4.2 in \cite{LinWeiCPAM2010}, we set
$$
\Phi(u, v, \tau)=\frac{v^\gamma}{(u^2 + v^2 + \tau)^{\beta}} +\frac{\varepsilon v}{u^2 + v^2},\quad \gamma\in (0, 1),\quad 2\beta-\gamma = \sigma.
$$
Then
\begin{equation*}
\begin{aligned}
&-\Phi_\tau + \Delta \Phi \\
&= v^{\gamma -2}\left(\tau+u^2+v^2\right)^{-\beta -2}\left(-\gamma\left(\tau+u^2+v^2\right)\left(\tau+u^2+(4\beta +1) v^2\right)\right)\\
&\quad + v^{\gamma-2}\left(\tau+u^2+v^2\right)^{-\beta -2}\left(\beta v^2 \left((4\beta +1)\left(u^2+v^2\right)-3\tau\right)+\gamma ^2\left(\tau+u^2+v^2\right)^2\right)\\
& < \beta(4\beta-4\gamma^2+1)v^{\gamma}\left(\tau+u^2+v^2\right)^{-\beta -1}\\
& = \beta(2\sigma+2\gamma-4\gamma^2+1)v^{\gamma}\left(\tau+u^2+v^2\right)^{-\beta -1} < 0.
\end{aligned}
\end{equation*}
if we choose $\sigma$ sufficiently small and $\gamma \in (0, 1)$ sufficiently close to 1. Then the function $\Phi(u, v,\tau +M)$ is a positive super-solution of equation (\ref{e:limiteequation}) in $\mathbb{R}^2_+\times [-M, 0]$. Hence $|\phi(u, v,\tau)| \le \Phi(u, v,\tau +M)$. Letting $M\to +\infty$ we have
$$
|\phi(u, v, \tau)|\leq\frac{\varepsilon v}{u^2 + v^2}.
$$
Since $\varepsilon$ is arbitrary, $\phi \equiv 0$.
\end{proof}
{\it Proof of Proposition \ref{proposition4.1}}.
Let $\phi$ be the unique solution of (\ref{e5:32}), from Lemma \ref{l5:3}, for any $\tau_1 > 0$, we have
\begin{equation*}
|\phi(u, v, \tau)|\leq C\tau^{-\nu}(1+|(u, v)|)^{-\sigma}\left(\|g\|_{2+\sigma, \tau_1} + \|h\|_{1+\sigma, \tau_1}\right).
\end{equation*}
Since $\|g\|_{a,v} < +\infty$, we get
\begin{equation*}
|g(u, v, \tau)|\leq C\tau^{-\nu}(1+|(u, v)|)^{-a}\|g\|_{a,\nu}
\end{equation*}
and
\begin{equation*}
\|g\|_{2+\sigma, \tau_1}\leq R^{2+\sigma-a}\|g\|_{a,\nu}.
\end{equation*}
Similarly, we have
\begin{equation*}
\|h\|_{1+\sigma, \tau_1}\leq R^{1+\sigma-a}\|h\|_{a,\nu}.
\end{equation*}
Therefore
\begin{equation*}
|\phi(u, v, \tau)|\leq C\tau^{-\nu}(1+|(u, v)|)^{-\sigma}\left(R^{2+\sigma-a}\|g\|_{a,\nu} + R^{1+\sigma-a}\|h\|_{a,\nu}\right).
\end{equation*}
\qed

\subsection{The whole linear theory.} Combine Propositions \ref{prop1} and \ref{proposition4.1}, we obtain the main result of this section.
\begin{prop} \label{prop1section4final}
Let $1< a < 2$, $\nu > 0$ be given positive numbers. Then, for any $g$, $h$ with $\|g\|_{a, \nu} <+\infty$, $\|h\|_{a, \nu} <+\infty$ and satisfying \begin{equation}\label{e:orthognalconditionz3final2}
\int_{B^+_{2R}}g(u, v, \tau)\cdot Z_2(u, v)dudv + \int_{-2R}^{2R}h(u,\tau)\cdot Z_2(u)du = 0\quad\text{for all }\tau\in (\tau_0,\infty)
\end{equation}
\begin{equation}\label{e:orthognalconditionz3final3}
\int_{B^+_{2R}}g(u, v, \tau)\cdot Z_3(u, v)dudv + \int_{-2R}^{2R}h(u,\tau)\cdot Z_3(u)du = 0\quad\text{for all }\tau\in (\tau_0,\infty)
\end{equation}
there exist $\phi = \phi[g, h]$ solving (\ref{e:innerproblemlinear}) which defines a bounded linear operators of $g$ and $h$. Furthermore, for some $\sigma\in (0, 1)$, we have the following estimate
\begin{equation*}
\begin{aligned}
&|\phi[g, h]|\lesssim\\
&\quad \tau^{-\nu}R^{2}\log^2 R\left(\frac{R^{2-a}\|h^0\|_{a,\nu}}{1+|(u, v)|}+\frac{R^{4-a}\|g^0\|_{a,\nu}}{(1+|(u, v)|)^{2}}\right)\\
&\quad +\tau^{-\nu}\left(\frac{R^{1+\sigma-a}\|h^1\|_{a, \nu}}{(1+|u|)^{\sigma}} + \frac{R^{2+\sigma-a}\|g^1\|_{a, \nu}}{(1+\left|(u, v)\right|)^{\sigma}}\right).
\end{aligned}
\end{equation*}
Here $g = g^0+g^1$, the first component of $g^0$ and the second component of $g^1$ are odd in the $u$ variable, the second component of $g^0$ and the first component of $g^1$ are even in the $u$ variable. We decompose $h = h^0 + h^1$ similarly.
\end{prop}
\begin{remark}
If conditions (\ref{e:orthognalconditionz3final2}) and (\ref{e:orthognalconditionz3final3}) are not satisfied, by the same argument of Step 1 in Proposition \ref{prop1}, we find a solution $\phi$ of (\ref{e:innerproblemlinear}) satisfying
\begin{equation*}
\begin{aligned}
(1+|(u, v)|)&|\nabla\phi(u, v, \tau)|+|\phi(u, v,\tau)|\\
&\lesssim\tau^{-\nu}R^2\log^2 R(R^{2-a}\|g\|_{a, \nu}+(1+|u|)^{1-a}\|h\|_{a, \nu}).
\end{aligned}
\end{equation*}
We will use this fact in Section 5.
\end{remark}
\section{Solving the inner-outer gluing system}
We separate the proof of Theorem \ref{t:main} into the following steps.

{\bf Step 1}. We formulate the inner-outer system (\ref{e:innerproblem})-(\ref{e:outerproblem}) into a fixed point problem in a suitable space.

\noindent $\bullet$ The inner problem. Define
\begin{equation*}
G_1[\lambda,\xi,\psi](u, v, \tau) = \chi_{\mathcal{D}_{2R}}\lambda^2\mathcal{E}^*_1\text{ in }\mathbb{R}^2_+\times (\tau_0, +\infty),
\end{equation*}
\begin{equation*}
\begin{aligned}
&G_2[\lambda, \xi, \psi](u, 0, \tau)\\
&= \chi_{\mathcal{D}_{2R}\cap\left(\partial\mathbb{R}^2_+\times (\tau_0, +\infty)\right)}\left(\lambda\Pi_{\omega^\perp}\mathcal{E}^*_2 + \frac{2}{1+\left|u\right|^2}\Pi_{\omega^\perp}\psi\right)\\
&\quad +\chi_{\mathcal{D}_{2R}\cap\left(\partial\mathbb{R}^2_+\times (\tau_0, +\infty)\right)}\times\\
&\left(\frac{1}{\pi}\int_{\mathbb{R}}\frac{(\omega(u,0)-\omega(u-z,0))\cdot\left(\Pi_{\omega^\perp}\psi(u,0,\tau) -\Pi_{\omega^\perp}\psi(u-z,0,\tau)\right)}{|z|^2}dz\right)\omega(u, 0),
\end{aligned}
\end{equation*}
\begin{equation*}
\begin{aligned}
c[\lambda, \xi, \psi](\tau) &= \frac{1}{\int_{B^+_{2R}}\chi Z_2^2 + \int_{-2R}^{2R}\chi Z_2^2}\times\\
&\quad\quad\quad\quad\quad\quad\left(\int_{B^+_{2R}}G_1[\lambda,\xi,\psi]\cdot Z_{2}dudv + \int_{-2R}^{2R}G_2[\lambda,\xi,\psi]\cdot Z_{2}du\right),
\end{aligned}
\end{equation*}
\begin{equation*}
\begin{aligned}
d[\lambda, \xi, \psi](\tau) &= \frac{1}{\int_{B^+_{2R}}\chi Z_3^2 + \int_{-2R}^{2R}\chi Z_3^2}\times\\
&\quad\quad\quad\quad\quad\quad\left(\int_{B^+_{2R}}G_1[\lambda,\xi,\psi]\cdot Z_{3}dudv + \int_{-2R}^{2R}G_2[\lambda,\xi,\psi]\cdot Z_{3}du\right),
\end{aligned}
\end{equation*}
\begin{equation*}
\overline{G}_1[\lambda,\xi,\psi](u, v, \tau) = c(\tau)\chi Z_2(u, v) + d(\tau)\chi Z_3(u, v),
\end{equation*}
\begin{equation*}
\overline{G}_2[\lambda,\xi,\psi](u, \tau) = c(\tau)\chi Z_2(u) + d(\tau)\chi Z_3(u).
\end{equation*}
Here $\chi(u, v) = \frac{1}{1+|(u, v)|}$. Then $\phi$ solves equation (\ref{e:innerproblem}) if $\phi_1$ and $\phi_2$ solve
\begin{equation*}
\left\{
\begin{array}{ll}
\partial_\tau\phi_1 = \Delta\phi_1 +(G_1-\overline{G}_1)[\lambda,\xi,\psi](u, v, \tau)\text{ in }\mathbb{R}^2_+\times (\tau_0, +\infty),\\\\
-\frac{d}{dv}\phi_1(u, 0, \tau) = \frac{2}{1+\left|u\right|^2}\phi_1\\
\quad + \frac{1}{\pi}\left[\int_{\mathbb{R}}\frac{(\omega(u,0)-\omega(u-z,0))\cdot\left(\phi_1(u,0,\tau) -\phi_1(u-z,0,\tau)\right)}{|z|^2}dz\right]\omega(u, 0)\\
\quad + (G_2-\overline{G}_2)[\lambda,\xi,\psi](u, 0, \tau)\text{ in }\partial\mathbb{R}^2_+\times (\tau_0, +\infty),\\
\phi_1 = 0\text{ in }\mathbb{R}^2_+\times(-\infty, \tau_0]
\end{array}
\right.
\end{equation*}
and
\begin{equation*}
\left\{
\begin{array}{ll}
\partial_\tau\phi_2 = \Delta\phi_2 + \overline{G}_1[\lambda,\xi,\psi](u, v, \tau)\text{ in }\mathbb{R}^2_+\times (\tau_0, +\infty),\\\\
-\frac{d}{dv}\phi_2(u, 0, \tau) = \frac{2}{1+\left|u\right|^2}\phi_2\\
\quad + \frac{1}{\pi}\left[\int_{\mathbb{R}}\frac{(\omega(u,0)-\omega(u-z,0))\cdot\left(\phi_2(u,0,\tau) -\phi_2(u-z,0,\tau)\right)}{|z|^2}dz\right]\omega(u, 0)\\
\quad + \overline{G}_2[\lambda,\xi,\psi](u, 0, \tau)\text{ in }\partial\mathbb{R}^2_+\times (\tau_0, +\infty),\\
\phi_2 = 0\text{ in }\mathbb{R}^2_+\times(-\infty, \tau_0],
\end{array}
\right.
\end{equation*}
respectively. Let $\phi=:\mathcal{T}[g, h]$ be the bounded linear operator constructed in Proposition \ref{prop1section4final}, then (\ref{e:innerproblem}) is equivalent to the following fixed point problem
\begin{equation*}
\left\{
\begin{aligned}
\phi_1&=\mathcal{T}[(G_1-\overline{G}_1)(\lambda,\xi,\psi), (G_2-\overline{G}_2)(\lambda,\xi,\psi)],\\
\phi_2&=\mathcal{T}[\overline{G}_1(\lambda,\xi,\psi), \overline{G}_2(\lambda,\xi,\psi)].
\end{aligned}
\right.
\end{equation*}

\noindent $\bullet$ The outer problem. Rewrite equation (\ref{e:outerproblem}) as
\begin{equation}\label{e:outerproblem5.1}
\left\{
\begin{array}{ll}
\partial_t\psi = \Delta\psi + H_1[\psi,\phi,\lambda, \xi](x, 0, t)\text{ in }\mathbb{R}^2_+\times (0, T), \\
-\frac{d}{dy} \psi(x, 0, t) = H_2[\psi,\phi,\lambda, \xi](x, 0, t)\text{ in }\partial\mathbb{R}^2_+\times (0, T),
\end{array}
\right.
\end{equation}
where
\begin{equation*}
\begin{aligned}
& H_1[\psi,\phi,\lambda, \xi](x, 0, t) =
- \frac{\dot{\lambda}}{\lambda}\eta(u, v)\cdot \nabla_{(u, v)}\phi - \eta\frac{\dot{\xi}}{\lambda}\frac{d\phi}{du}\\
&\quad\quad\quad + [\Delta\eta\phi + \frac{2}{\lambda}\nabla\eta\nabla\phi - \partial_t\eta\phi] + (1-\eta)\mathcal{E}^*_1,
\end{aligned}
\end{equation*}
\begin{equation*}
\begin{aligned}
& H_2[\psi,\phi,\lambda, \xi](x, 0, t) =
(1-\eta)\frac{2\lambda}{\lambda^2+|x-\xi(t)|^2}\psi + (1-\eta)\Pi_{U^\perp}\mathcal{E}^*_2\\
& \quad\quad\quad\quad\quad\quad\quad\quad\quad\quad + \left(\frac{d}{dy}\eta\right)(x, 0, t) \phi + N_U(\Pi_{U^\perp}(\Phi^*+\varphi)).
\end{aligned}
\end{equation*}
To solve (\ref{e:outerproblem5.1}), we first consider the corresponding linear problem
\begin{equation}\label{e:outerproblemmodel}
\left\{
\begin{array}{lll}
\psi_t = \Delta\psi + f(x, y, t)\text{ in }\mathbb{R}^2_+\times (0, T),\\
-\frac{d}{dy}\psi = g(x,t), x\in \mathbb{R}\times (0, T),\\
\psi(q,0,T) = 0, \\
\psi(x,y, 0) = (c_1{\bf e_1} + c_2 {\bf e_2})\eta_1\text{ in }\mathbb{R}^2_+.
\end{array}
\right.
\end{equation}
for suitable constants $c_1,c_2$, where
\begin{equation*}
\mathbf{e_1}  = \left ( \begin{matrix} 1\\ 0\end{matrix}  \right ),\quad
\mathbf{e_2}  = \left ( \begin{matrix} 0\\ 1\end{matrix}  \right ),
\end{equation*}
and $\eta_1$ is a smooth cut off function with compact support and $\eta_1\equiv 1$ in a neighborhood of $(q, 0)$. For a function $f(x, y, t)$,
define the $L^\infty$-weighted norm as follows
$$
\|f\|_{**} : = \sup_{\mathbb{R}^2_+\times(0,T)}\Big(1+\sum_{i=1}^{3}\varrho_i(x, y, t)\, \Big )^{-1}{|f(x, y, t)|}.
$$
Here $\varrho_1 :=\lambda_0^{\Theta-2}R^{-a}\chi_{\{r<2R\lambda_0\}}$, $\varrho_2:=T^{-\sigma_0}(1-\eta)\frac{\lambda_0}{r^2+\lambda_0^2}$,
$\varrho_3 :=1$, $\sigma_0$ and $\Theta>0$ are small. Also, for $\gamma\in (0, \frac{1}{2})$, we define
\begin{equation}\label{defNorm1Psi}
\begin{aligned}
\|\psi\|_{a, \Theta, \gamma}&=\frac{1}{\lambda_0(0)^{\Theta}R(0)^{2-a}|\log T|}\sup_{\mathbb{R}^2_+\times(0,T)}|\psi(x,y,t)|\\
&\quad+\sup_{\mathbb{R}^2_+\times(0,T)}\frac{1}{\lambda_0(t)^{\Theta} R(t)^{2-a}|\log(T-t)|}|\psi(x,y,t)-\psi(x,y,T)|\\
&\quad+\sup_{\mathbb{R}^2_+\times(0,T)}\frac{1}{\lambda_0(t)^{\Theta-1}R(0)^{1-a}}|\nabla\psi(x,y,t)|\\
&\quad+\sup_{\mathbb{R}^2_+\times(0,T)}\frac{1}{\lambda_0(t)^{\Theta-1}R(t)^{1-a}}|\nabla \psi(x,y,t)-\nabla\psi(x,y,T)|\\
&\quad+\sup_{\mathbb{R}^2_+\times(0,T)}\frac{1}{\lambda_0(t)^{\Theta-1-2\gamma}R(t)^{1-a-2\gamma}}\frac{|\nabla\psi(x,y,t)-\nabla\psi(x',y',t)|}{|(x,y)-(x',y')|^{2\gamma}}\\
&\quad+\sup_{}\frac{1}{\lambda_0(t)^{\Theta-1-2\gamma}R(t)^{1-a-2\gamma}}\frac{|\nabla\psi(x,y,t_2)-\nabla\psi(x,y,t_1)|}{(t_2-t_1)^{\gamma}},
\end{aligned}
\end{equation}
where the last supremum is taken over $(x, y)\in\mathbb{R}^2_+$, $0\leq t_1< t_2\leq T$ and $t_2-t_1\leq\frac{1}{10}(T-t_2)$.
Then by minor modifications of \cite{davila2017singularity}, we have
\begin{prop}\label{prop3}
For $T$, $\varepsilon>0$, there exists a linear operator mapping functions $f:\mathbb{R}^2_+\times (0,T)\to\mathbb{R}^2$, $g:\partial\mathbb{R}^2_+\times(0,T)\to\mathbb{R}^2$ with $\|f\|_{**}<\infty$, $\|g\|_{**}<\infty$ into $\psi$, $c_1,c_2$ so that (\ref{e:outerproblemmodel}) is satisfied and the following estimate holds
\begin{equation*}
\|\psi\|_{a, \Theta, \gamma}\leq C\left(\|f\|_{**}+\|g\|_{**}\right).
\end{equation*}
\end{prop}
\noindent Let $\psi=\mathcal{S}[f, g]$ be the operator defined in Proposition \ref{prop3}, then (\ref{e:outerproblem5.1}) is equivalent to
\begin{equation*}
\begin{aligned}
\psi = \mathcal{S}[H_1, H_2, \psi_{\infty}](\phi, \psi, \lambda, \xi).
\end{aligned}
\end{equation*}

\noindent $\bullet$ The choice of $\lambda$. To make $d(\tau)$ as small as possible, we solve the following equation approximately,
\begin{equation}\label{e:lambda}
\int_{B^+_{2R}}G_1\cdot Z_3 dudv + \int_{-2R}^{2R}G_2\cdot Z_3 du = 0.
\end{equation}
This is the case $\alpha = 0$ of $\lambda$-$\alpha$ system in \cite{davila2017singularity}, hence (\ref{e:lambda}) is equivalent to the fixed point problem
\begin{equation*}
\lambda = \mathcal{A}_1(\psi, \lambda, \xi).
\end{equation*}
We refer the readers to \cite{davila2017singularity} for details.

\noindent $\bullet$ The choice of $\xi$. To make $c(\tau)$ as small as possible, we solve the following equation
\begin{equation}\label{e:xi}
\int_{B^+_{2R}}G_1\cdot Z_2 dudv + \int_{-2R}^{2R}G_2\cdot Z_2 du = 0
\end{equation}
which is equivalent to a nonlinear ODE for form
\begin{equation*}
\dot{\xi} = \frac{1}{\int_{B^+_{2R}}Z_2\cdot Z_2 dudv}\left(\int_{B^+_{2R}}\left(G_1 + \dot{\xi}Z_2\right)\cdot Z_2 dudv + \int_{-2R}^{2R}G_2\cdot Z_2 du\right).
\end{equation*}
This can be rewritten as a fixed point problem
\begin{equation*}
\xi = \mathcal{A}_2(\psi, \lambda, \xi).
\end{equation*}

Combine the above arguments, the inner-outer system (\ref{e:innerproblem})-(\ref{e:outerproblem}) is equivalent to the following fixed point problem
\begin{equation}\label{defMapF1}
\psi=\mathcal{S}[H_1, H_2, \psi_{\infty}](\psi, \phi, \lambda, \xi),
\end{equation}
\begin{equation}\label{defMapF2}
\phi_1=\mathcal{T}[(G_1-\overline{G}_1)(\lambda,\xi,\psi), (G_2-\overline{G}_2)(\lambda,\xi,\psi)],
\end{equation}
\begin{equation}\label{defMapF20}
\phi_2=\mathcal{T}[\overline{G}_1(\lambda,\xi,\psi), \overline{G}_2(\lambda,\xi,\psi)],
\end{equation}
\begin{equation}\label{defMapF3}
\lambda = \mathcal{A}_1(\psi, \lambda, \xi),\\
\end{equation}
\begin{equation}\label{defMapF4}
\xi = \mathcal{A}_2(\psi, \lambda, \xi).
\end{equation}

{\bf Step 2}. To set up the fixed point problem (\ref{defMapF1})-(\ref{defMapF4}), we give a description of the relevant functional space . First, set
\begin{equation*}
R(t) = \lambda_0(t)^{-\beta}, \,\, \beta = \frac{1}{4}+\sigma
\end{equation*}
and
$$
a = 2- \sigma,
$$
for a small but fixed number $\sigma>0$. Take $\phi$ in the following space
$$
X(a,\nu) = \{\phi\in C(\overline{\mathcal{D}}_{2R}):\,\nabla\phi\in C(\overline{\mathcal{D}}_{2R}), \ \|\phi\|_{X(a,\nu)}<\infty\},
$$
where
\begin{equation*}
\begin{aligned}
\|\phi\|_{X(a,\nu)}=\sup_{(u, v,\tau) \in \mathcal D_{2R}}\frac{1}{\lambda_0^{\nu}\frac{R^{6-a}\log^2 R}{1+\left|(u, v)\right|^2}}\left[(1+\left|(u, v)\right|)\left|\nabla\phi(u, v,\tau)\right|+|\phi(u, v,\tau)|\right],
\end{aligned}
\end{equation*}
$a\in (1,2)$ is close to 2 and $\nu \in (0,1)$ is close to 1. Also we take $\psi$ in the space
$$
Y(a',\Theta,\gamma') = \{\psi\in C(\mathbb{R}^2_+\times [0,T)):\nabla \psi\in C(\mathbb{R}^2_+\times [0,T)),\|\psi\|_{a',\Theta,\gamma'} <\infty\}
$$
for parameters $a'<a$ and $\gamma'<\gamma$. Note that the norm $\|\cdot\|_{a',\Theta,\gamma'}$ is weaker than $\|\cdot\|_{a,\Theta,\gamma}$,
and the inclusion $Y(a,\nu,\gamma)\hookrightarrow Y(a',\nu,\gamma')$ is compact.

We assume that the parameter $\lambda$ is in the space of $C^1[-T, T]$ functions satisfying $\lambda(T) = 0$ with norm
$$
\|g\|_{\mu} = \sup_{t\in [-T, T]}(T-t)^{-\mu}|g(t)|
$$
for $\mu\in (0, 1)$ small, while $\xi$ is in the space $C^1([0,T])$ satisfying $\xi(T)=q$ with norm
$$
\|\xi\|_{\sigma}=\sup_{t\in [0,T]}(T-t)^{-\sigma}|\dot\xi(t)| ,
$$
for some $\sigma>0$ fixed.

For $R_1 >0 $ small but fixed, let us define the set
\begin{equation*}
\begin{aligned}
A = \{(\psi, &\phi_1, \phi_2, \lambda, \xi)\in Y(a',\Theta,\gamma')\times X(a,\nu)\times X(a,\nu)\times C^1[-T, T]\times C^1[0, T]\Big |\\
&\|\phi_1\|_{X(a,\nu)}+\|\phi_2\|_{X(a,\nu)}+\|\psi\|_{Y(a',\Theta,\gamma')}+\|\lambda\|_{\mu}+\|\xi\|_{\sigma}\leq T^{\sigma_0}+R_1\}.
\end{aligned}
\end{equation*}
Let $(\psi, \phi, \lambda, \xi)\mapsto\mathcal F(\psi, \phi, \lambda, \xi)$ be the map defined by (\ref{defMapF1})-(\ref{defMapF4}).

{\bf Step 3}. We show that $\mathcal F$ maps the set $A$ into itself and it is a compact operator. To this aim, we should estimate (\ref{defMapF1})-(\ref{defMapF4}) respectively.

\noindent {\bf Estimations for (\ref{defMapF1})}.
We claim if $R_1>0$ is fixed and small, there holds
\begin{equation}\label{estH}
\|H_1(\psi,\phi,\lambda,\xi)\|_{**}\leq C(T^{\sigma_0}+R_1).
\end{equation}
First we consider the term $(1-\eta)\frac{2\lambda}{\lambda^2+|x-\xi(t)|^2}\psi$. Since $\psi(q,0,T)=0$ and from the definition of $\| \ \|_{a',\Theta,\gamma'}$, we have
\begin{equation*}
\begin{aligned}
|\psi(x,0,t)|&\leq(|\psi(x,0,t)-\psi(x,0,T)|+|\psi(x,0,T)-\psi(q,0,T)|\\
&\leq (r+\lambda_0^\Theta(t)R(t)^{2-a'}|\log(T-t)|)\|\psi\|_{a',\Theta,\gamma'}.
\end{aligned}
\end{equation*}
Hence
\begin{equation*}
\begin{aligned}
&|(1-\eta)\frac{2\lambda}{\lambda^2+|x-\xi(t)|^2}\psi|\\
&\leq (1-\eta)\frac{2\lambda_0}{(r+\lambda_0)^2}(|\psi(x,0,t)-\psi(x,0,T)|+|\psi(x,0,T)-\psi(q,0,T)| )\\
&\leq C(1-\eta)\frac{\lambda_0}{(r+\lambda_0)^2}(r+\lambda_0^\Theta(t) R(t)^{2-a'}|\log(T-t)|)\|\psi\|_{a',\Theta,\gamma'}\\
&\leq C(\varrho_2+\varrho_3)\|\psi\|_{a',\Theta,\gamma'}
\end{aligned}
\end{equation*}
and
\begin{equation*}
\|(1-\eta)\frac{2\lambda}{\lambda^2+|x-\xi(t)|^2}\psi\|_{**}\leq C\|\psi\|_{Y(a',\nu,\gamma')}.
\end{equation*}
Next we consider $\Delta\eta\phi$. From the definition of $\|\cdot\|_{X(a,\nu)}$, when $R\leq|(u, v)|\leq 2R$, we have
$$
|\phi(u, v, \tau)|+(1+|(u, v)|)|\nabla \phi(u, v, \tau)|\leq\|\phi\|_{X(a,\nu)}\lambda_0^\nu R^{4-a}\log^2R.
$$
Hence
\begin{equation*}
\begin{aligned}
|\Delta\eta\phi|&\leq\frac{1}{\lambda^2 R^2}\chi_{\{|(x,y)-(\xi,0)|\leq 2\lambda R\}}|\phi(y,\tau)|\\
&\leq C\lambda_0^{\nu-2}R^{2-a}\log^2R\chi_{\{|(x,y)-(q,0)|\leq C\lambda_0 R\}}\|\phi\|_{X(a,\nu)}\\
&\leq C\lambda_0^{\Theta-2}R^{-a'}\chi_{\{|(x,y)-(q,0)|\leq C\lambda_0 R\}}\|\phi\|_{X(a,\nu)}\\
&\leq C\varrho_1\|\phi\|_{X(a,\nu)}
\end{aligned}
\end{equation*}
and
\begin{equation*}
\|\Delta\eta\phi\|_{**}\leq C\|\phi\|_{X(a,\nu)}.
\end{equation*}
Similarly, we have
$$
\|\left(\frac{d}{dy}\eta\right)(x, 0, t)\phi\|_{**}+\|(\partial_t\eta)\phi \|_{**}+\|\lambda^{-1}\nabla\eta\nabla\phi\|_{**}\leq C\|\phi\|_{X(a,\nu)}.
$$
For the term $-\frac{\dot{\lambda}}{\lambda}\eta(u, v)\cdot\nabla_{(u, v)}\phi-\frac{\dot{\xi}}{\lambda}\eta\frac{d\phi}{du}$, since $|\dot\lambda|\leq C$, we have
\begin{equation*}
\left|\frac{\dot{\lambda}}{\lambda}\eta(u, v)\cdot\nabla_{(u, v)}\phi\right|\leq C\lambda_0^{\nu-1}R^{6-a}\log^2R\chi_{\{|x-q|\leq 2\lambda_0(t)R(t)\}}\|\phi\|_{X(a,\nu)}.
\end{equation*}
Similarly,
\begin{equation*}
\left|\frac{\dot{\xi}}{\lambda}\eta\frac{d\phi}{du}\right|\leq C\lambda_0^{\nu-1}R^{6-a}\log^2R\chi_{\{|x-q|\leq 2\lambda_0(t)R(t)\}}\|\phi\|_{X(a,\nu)}.
\end{equation*}
Therefore,
\begin{equation*}
\|-\frac{\dot{\lambda}}{\lambda}\eta(u, v)\cdot\nabla_{(u, v)}\phi-\frac{\dot{\xi}}{\lambda}\eta\frac{d\phi}{du}\|_{**}\leq C\|\phi\|_{X(a,\nu)}.
\end{equation*}
For $(1-\eta)\Pi_{U^\perp}\mathcal{E}^*_2$, we have
\begin{equation*}
|(1-\eta)\Pi_{U^\perp}\mathcal{E}^*_2|\leq (1-\eta)\frac{\lambda_0}{r^2+\lambda_0^2},
\end{equation*}
hence
$$
\|(1-\eta)\Pi_{U^\perp}\mathcal{E}^*_2\|_{**}\leq C T^{\sigma_0}.
$$
Similarly,
$$
\|(1-\eta)\mathcal{E}^*_1\|_{**}\leq C T^{\sigma_0}.
$$
The proof of the estimate
$$
\|N(\varphi^*+\psi+\eta Q_\alpha\phi)\|_{**}\leq C(\|\psi\|_{a',\Theta,\gamma'}+\|\phi\|_{X(a,\nu)})
$$
is analogous as the previous terms, so we omit the details. From the above estimates, we obtain (\ref{estH}).
By (\ref{estH}) and Proposition (\ref{prop3}), there holds
\begin{equation}\label{estS}
\|\mathcal{S}[H_1, H_2, \psi_{\infty}](\phi, \psi, \lambda, \xi)\|_{a',\Theta,\gamma'}\leq C R_1.
\end{equation}

\noindent{\bf Estimations for (\ref{defMapF2})}. Now we consider (\ref{defMapF2}). By Lemma 3.2 in \cite{SWY}, there holds
\begin{equation*}
\begin{aligned}
|\tilde L_w[\psi]|&\leq C\frac{\lambda_0}{1+\rho^2}\|\nabla\psi\|_{L^\infty}\leq C\frac{\lambda_0}{1+\rho^2}T^{\Theta}\|\psi\|_{Y(a',\nu,\gamma')}.
\end{aligned}
\end{equation*}
Fix $a_1\in(a,2)$ and $\nu_1\in(\nu,1)$, which implies
$$
\|G_1(\lambda,\xi,\psi)\|_{a_1,\nu_1}\leq C T^{1-\nu_1} + C T^{\Theta}R_1
$$
and
$$
\|G_2(\lambda,\xi,\psi)\|_{a_1,\nu_1}\leq CT^{1-\nu_1}+CT^{\Theta}R_1.
$$
Therefore
\begin{equation}\label{estT1}
\begin{aligned}
\|\mathcal{T}[G_1(\lambda,\xi,\psi), G_2(\lambda,\xi,\psi)]\|_{X(a_1,\nu_1)}\leq CT^{1-\nu_1} + CT^{\nu-1+\beta(a'-1)}R_1.
\end{aligned}
\end{equation}

\noindent{\bf Estimations for (\ref{defMapF20})}. From the choice of $\xi$, $c(\tau)=0$ and from the result of \cite{davila2017singularity},
\begin{equation*}
|d(\tau)|\leq C\lambda_0(T-t)^{\sigma_2}R(t)^{1-a'}(\|a(\cdot)-a(T)\|_{\mu,l-1} + \|a(\cdot)-a(T) \|_{\gamma',m,l-1})
\end{equation*}
and
\begin{equation*}
\|a(\cdot)-a(T)\|_{\mu,l-1} + \|a(\cdot)-a(T)\|_{\gamma',m,l-1}\leq C\|\psi\|_{a',\Theta,\gamma'}\leq C R_1.
\end{equation*}
Hence we have
\begin{equation}\label{estTildeSi00}
\|\tilde\phi_2\|_{X(a_1,\nu_1)}=\mathcal{T}[c_{lj}(\Theta(\lambda,\alpha,\xi,\psi),\psi))\chi Z_{lj}]\leq C R_1,
\end{equation}
which holds since the decay of $\chi Z_{3}$ is $\frac{1}{1+\rho^2}$ and $\nu_1$ is close to $\nu$ depending on $\sigma$.

\noindent{\bf Estimations for (\ref{defMapF3})}. From Proposition 6.1 in \cite{davila2017singularity}, we obtain
\begin{equation}\label{estTildeSi11}
\|\lambda\|_{\mu}\leq C\left|\log T\right|^{\beta(a'-1)}
\end{equation}
for $\mu = \beta(a-1)$.

\noindent{\bf Estimations for (\ref{defMapF4})}.
The definition of $\|\psi\|_{a',\Theta,\gamma'}$ implies that
\begin{equation}\label{estTildeSi}
\|\xi\|_{\mu_1} \leq C R_1
\end{equation}
for $\mu_1 = \nu-1+\beta(a-1)$.

From the estimates (\ref{estS}), (\ref{estT1}), (\ref{estTildeSi00}), (\ref{estTildeSi11}), (\ref{estTildeSi}) and standard parabolic estimates, $\mathcal F$ is compact from the set $A$ into itself. The existence of a solution follows then from Schauder's fixed point theorem, which completes the proof of Theorem \ref{t:main}.

\section*{Acknowledgements}
J. Wei is partially supported by NSERC of Canada. Y. Zheng is partially supported by NSF of China (11301374). Y.S. is partially supported by the Simons foundation.

\end{document}